\documentclass[a4paper]{article}

\usepackage{amsfonts}
\usepackage{amsthm}
\usepackage{amsmath}
\usepackage{amssymb}
\usepackage{xcolor}
\usepackage{tikz-cd}
\usepackage{enumitem}
\usepackage{scalerel}

\usepackage[margin=2.54cm]{geometry}

\numberwithin{equation}{section}
\newtheorem{theorem}{Theorem}[section]
\newtheorem*{theorem_A}{Theorem~A}
\newtheorem*{theorem_B}{Theorem~B}
\newtheorem{lemma}[theorem]{Lemma}
\newtheorem{proposition}[theorem]{Proposition}

\newtheorem{cor}[theorem]{Corollary}
\newtheorem*{cor*}{Corollary}
\theoremstyle{remark}
\newtheorem{remark}[theorem]{Remark}

\DeclareMathOperator{\Ric}{Ric}
\DeclareMathOperator{\Ad}{Ad}
\DeclareMathOperator{\vol}{vol}

\DeclareMathOperator{\II}{II}
\DeclareMathOperator{\Id}{Id}
\newcommand{\scperp}{\kern-0.5pt\scaleto{\perp}{4pt}\kern-0.5pt}
\newcommand{\scparal}{\kern-0.5pt\scaleto{/}{4pt}\mkern-1mu\scaleto{/}{4pt}}

\title{Integral Gauss formula and the Poisson equation for the $G_2$-Laplacian}

\author{Timothy Buttsworth\thanks{School of Mathematics and Statistics, The University of New South Wales,
Sydney, NSW 2052, Australia}~\thanks{Timothy Buttsworth is the recipient of an Australian Research Council Discovery Early-Career Researcher Award (DE220100919) funded by the Australian Government.} \\
\small{\texttt{t.buttsworth@unsw.edu.au}}
\and Stepan Hudecek\thanks{School of Mathematics and Physics, The University of Queensland, St Lucia,~QLD 4072, Australia} \\
\small{\texttt{s.hudecek@uq.edu.au}}
\and Artem Pulemotov\footnotemark[3]~\thanks{Research supported by the Australian Government through the Australian Research Council's Discovery Projects funding scheme (DP220102530 and DP240101772).} \\
\small{\texttt{a.pulemotov@uq.edu.au}}}

\begin{document}

\maketitle

\begin{abstract}
We produce a formula, analogous to the Gauss-Codazzi equation, which relates the geometry of a $G_2$-structure and its Hodge Laplacian to the geometry of the induced $SU(3)$-structure on an embedded hypersurface. As an application, we obtain necessary conditions for the solvability of the Poisson equation for (not necessarily closed) $G_2$-structures in a neighbourhood of this hypersurface. Next, we prove that our conditions are sufficient in the cohomogeneity one setting, assuming the symmetry group is compact and simple.
\end{abstract}

\tableofcontents

\section{Introduction and the main results}\label{sec_intro}


Let $M$ be a 7-dimensional orientable spin manifold. Given a $G_2$-structure~$\phi$ on~$M$, denote by $\Delta_\phi$ the Hodge Laplacian of the metric $g$ induced by~$\phi$. The present paper studies the nonlinear operator
\begin{align}\label{G2_Laplacian}
\phi\mapsto\Delta_\phi\phi
\end{align}
and the associated Poisson equation. We refer to this operator as the $G_2$-Laplacian. It plays a prominent role in $G_2$-geometry, particularly via the analysis of the associated heat flow; see the survey~\cite{L20}. Our first objective is to produce an integral Gauss formula for the operator~\eqref{G2_Laplacian}, linking it to the natural Laplacian on a compact orientable hypersurface~$\Sigma$ in~$M$. Once this is achieved, we demonstrate that our formula provides an integrability condition for the Poisson equation associated with~\eqref{G2_Laplacian} in the case where $M$ is a cohomogeneity one space
and $\Sigma$ is a principal orbit.

A $G_2$-structure $\phi$ on $M$ yields an $SU(3)$-structure $(\omega,\Omega)$ on~$\Sigma$; see, e.g.,~\cite{MC06}. This observation has led to numerous advances in $G_2$-geometry. Examples include Hitchin's construction of $G_2$-manifolds from half-flat $SU(3)$-structures and its extensions (see, e.g.,~\cite{FT08,CLSSH11,CN21,S24} and references therein), new constructions of mirror dual Calabi--Yau manifolds (see~\cite{AS10} and also~\cite{BM03}), improved understanding of 7-manifolds with boundary~\cite{D17,D18}, the discovery of a new flow of $SU(3)$-structures~\cite{F22} and the development of connections to spin geometry~\cite{ACFH15}. As shown by Mart\'in Cabrera~\cite{MC06}, one can establish strong links between $\phi$ and $(\omega,\Omega)$ using the shape operator of the hypersurface~$\Sigma$. The geometry of $M$ is related to that of $\Sigma$ via the Gauss--Codazzi equations. Developing these connections further, we find a link between the $G_2$-Laplacian~\eqref{G2_Laplacian} and the natural Laplacian on~$\Sigma$ in the spirit of the classical Gauss formula. This is our first main result, presented as Theorem~A below. While results of this nature are available for the Hodge Laplacian~\cite{CC25}, our formula appears to be the first obtained in the nonlinear framework of $G_2$-geometry.

To state Theorem~A, we need additional notation. In what follows, we choose a background metric~$g_0$ and an orientation on~$M$. Let $d$ and $\delta$ be the exterior derivative on $M$ and the codifferential with respect to~$g_0$. Denote by $i$ the embedding $\Sigma\hookrightarrow M$. The metric $g_0$ induces a metric $g_0^\Sigma=i^*g_0$. Fix an orientation on~$\Sigma$. We will use the same notation $\langle\cdot,\cdot\rangle$ for the scalar products with respect to $g_0$ and~$g_0^\Sigma$. This will not cause confusion. Let $d^{\Sigma}$ be the exterior derivative on~$\Sigma$. Finally, denote by $\delta^\Sigma$ and $\Delta^\Sigma$ the codifferential and the Hodge Laplacian associated with~$g_0^\Sigma$.

To obtain the formula in Theorem~A, we must integrate over~$\Sigma$. This makes our result reminiscent of integral formulas for the scalar curvature of a manifold with a $G_2$-structure; see~\cite{MC22} and references therein. In what follows, integration is always carried out with respect to the volume form of~$g_0^\Sigma$, and $\|\cdot\|_{L^2}$ denotes the $L^2$-norm on the sections of the tensor bundle over~$\Sigma$. The pointwise variant of the formula in Theorem~A appears in Proposition~\ref{prop_pwise_Gauss}.

Fix a unit normal vector field $\mathbf n$ on~$\Sigma$ with respect to~$g_0$, and let $\nu=\mathbf{n}^\flat$ be its dual form field. On~$\Sigma$, we can decompose every $\varphi\in\Omega^k(M)$ for $k=1,\ldots,7$ into the sum
\begin{align*}
\varphi=\nu\wedge\varphi_{\scperp}+\varphi_{\scparal},
\end{align*}
where $\varphi_{\scperp}=\iota_\mathbf n\varphi$ and $\iota_{\mathbf n}\varphi_{\scparal}=0$. It will be convenient for us to identify the forms $\varphi_{\scperp}$ and $\varphi_{\scparal}$ with the forms they induce on~$\Sigma$. This means that $\varphi_{\scperp}=i^*(\iota_\mathbf n\varphi)\in\Omega^{k-1}(\Sigma)$ and $\varphi_{\scparal}=i^*\varphi\in\Omega^k(\Sigma)$.

\begin{theorem_A}
Consider a manifold $M$ with a $G_2$-structure~$\phi$ and a compact orientable embedded hypersurface $\Sigma$ in~$M$. Choose the background metric $g_0$ to coincide with the metric $g$ induced by~$\phi$. The $G_2$-Laplacian satisfies the formula
    \begin{equation}\label{Gauss_formula}
    \begin{split}
        \int_\Sigma \langle\Delta_\phi\phi,\phi\rangle = \int_\Sigma\langle \Delta^\Sigma\phi_{\scparal},\phi_{\scparal}\rangle -\int_\Sigma\langle\Delta^\Sigma\phi_{\scperp},\phi_{\scperp}\rangle &+2\int_\Sigma\langle (\Delta_\phi\phi)_{\scperp},\phi_{\scperp}\rangle \\
        &+ \big\|(d\phi)_{\scperp}+d^\Sigma\phi_{\scperp}\big\|^2_{L^2} - 
        \big\|(\delta\phi)_{\scparal} - \delta^\Sigma\phi_{\scparal}\big\|^2_{L^2}.
    \end{split}        
    \end{equation}
\end{theorem_A}

Exploiting the definition of the Hodge Laplacian and the compactness of $\Sigma$, we can transform formula~\eqref{Gauss_formula} as in the following corollary. If the $G_2$-structure $\phi$ is closed, which is the case of interest in many papers involving the operator~\eqref{G2_Laplacian}, then several terms disappear.

\begin{cor}\label{cor_ThmA}
Within the conditions of the theorem,
\begin{equation*}
\begin{split}
\int_\Sigma \langle\Delta_\phi\phi,\phi\rangle -2\int_\Sigma\langle (\Delta_\phi\phi)_{\scperp},\phi_{\scperp}\rangle 
= \big\|d^\Sigma\phi_{\scparal}\big\|_{L^2}^2&+\big\|\delta^\Sigma\phi_{\scparal}\big\|_{L^2}^2-\big\|d^\Sigma\phi_{\scperp}\big\|_{L^2}^2-\big\|\delta^\Sigma\phi_{\scperp}\big\|_{L^2}^2
\\
&+ \big\|(d\phi)_{\scperp}+d^\Sigma\phi_{\scperp}\big\|^2_{L^2} - 
\big\|(\delta\phi)_{\scparal} - \delta^\Sigma\phi_{\scparal}\big\|^2_{L^2}.
\end{split}
\end{equation*}
If $d\phi=0$, then
\begin{equation*}
\begin{split}
\int_\Sigma \langle\Delta_\phi\phi,\phi\rangle -2\int_\Sigma\langle (\Delta_\phi\phi)_{\scperp},\phi_{\scperp}\rangle 
= \big\|\delta^\Sigma\phi_{\scparal}\big\|_{L^2}^2-\big\|\delta^\Sigma\phi_{\scperp}\big\|_{L^2}^2
- 
\big\|(\delta\phi)_{\scparal} - \delta^\Sigma\phi_{\scparal}\big\|^2_{L^2}.
\end{split}
\end{equation*}
\end{cor}

Recall the Gauss--Codazzi equation
\begin{align}\label{GaussCodazzi}
S=S^\Sigma+2\Ric(\mathbf n,\mathbf n)+\|\II\|^2-H^2,
\end{align}
where $S$ and $S^\Sigma$ are the scalar curvatures of $M$ and~$\Sigma$, $\II$ is the second fundamental form of~$\Sigma$, and $H$ is the mean curvature. Remarkably, formula~\eqref{Gauss_formula} bears several commonalities with~\eqref{GaussCodazzi}, such as the coefficient 2 at the normal component of $\Delta_\phi\phi$ and the signs preceding the squared norms. We will reveal another connection between~\eqref{Gauss_formula} and~\eqref{GaussCodazzi} below.

Given a 3-form $\eta$ on the manifold~$M$, consider the Poisson equation
\begin{align}\label{Poisson_equation}
\Delta_\phi\phi=\eta.
\end{align}
The analysis of this equation is an essential step towards understanding the $G_2$-Laplacian~\eqref{G2_Laplacian}, particularly its image. We speculate that it is possible to gain information about the heat flow of~\eqref{G2_Laplacian} and to construct $G_2$-structures with various desirable properties via iterative procedures involving~\eqref{Poisson_equation}. For example, suppose that one can prove the existence and the convergence of a sequence $(\phi_n)_{n=1}^\infty$ such that
\begin{align*}
\Delta_{\phi_n}\phi_n=\phi_{n-1},\qquad n=2,3,\ldots.
\end{align*}
If the limit is a $G_2$-structure, then it must satisfy the eigenform equation for~\eqref{G2_Laplacian}, which suggests that it may be related to Laplacian solitons or nearly parallel $G_2$-structures. Similar ideas underpin the concept of the Ricci iteration introduced in~\cite{R07}. The study of this concept has yielded numerous insights into the Ricci flow and the Einstein equation; see, e.g.,~\cite{DR19,PR19,BH21}. The first- and last-named authors recently established the solvability of~\eqref{Poisson_equation} in a neighbourhood of a point assuming that the 3-form $\eta$ on the right-hand side was positive and closed~\cite{BP25}. In fact, they proved the existence of two closed $G_2$-structures that both satisfied~\eqref{Poisson_equation} but induced different orientations on~$M$. Subsequently, the second-named author investigated the global solvability of~\eqref{Poisson_equation} on homogeneous spheres~\cite{H25}. He obtained a series of existence theorems, along with uniqueness and non-uniqueness results.

The present paper studies solutions to~\eqref{Poisson_equation} in a neighbourhood of a hypersurface~$\Sigma\subset M$, particularly in the case where $M$ is a cohomogeneity one manifold and $\Sigma$ is a principal orbit. In order to classify these solutions, we impose Cauchy conditions to be satisfied on~$\Sigma$. More precisely, choose a $G_2$-structure~$\psi$ on~$M$. Assume that the background metric $g_0$ coincides with the metric induced by~$\psi$. Let $\psi^+$ and $\psi^-$ be fixed differential forms on $\Sigma$ of degrees 3 and 2, respectively. We demand that
\begin{align}\label{initial_conditions}
\phi|_{\Sigma}=\psi|_{\Sigma},\qquad
(d\phi)_{\scperp}=\psi^+,\qquad (\delta\phi)_{\scparal}=\psi^-,
\end{align}
where $|_\Sigma$ denotes the restriction to the pullback bundle $i^*(\Lambda^3(TM))$.
These formulas are a combination of absolute and relative boundary conditions on differential forms that generalise classic Neumann and Dirichlet boundary conditions on scalar functions and ensure the self-adjointness of the Hodge Laplacian. Notably, absolute and relative conditions have arisen previously in $G_2$-geometry and other nonlinear theories;
see, e.g.,~\cite{D18,M92,P08,CG13}.

Following an argument by Bryant, we prove in Proposition~\ref{prop_1st_order_cond} that
\begin{equation} \label{first order condition}
\langle(d\phi)_{\scperp},\phi_{\scparal}\rangle = -\langle(\delta\phi)_{\scparal},\phi_{\scperp}\rangle.
\end{equation}
This formula and Theorem~A readily provide necessary conditions on $\eta$, $\psi$, $\psi^+$ and $\psi^-$ for the solvability of the Cauchy problem~\eqref{Poisson_equation}--\eqref{initial_conditions}. Our second main result, which is the core of Theorem~B below, asserts that these conditions are also sufficient in the cohomogeneity one setting. More precisely, let $G$ be a compact connected simple Lie group acting on~$M$. If the principal orbits of this action are hypersurfaces in~$M$, then we call $M$ a cohomogeneity one manifold. The analysis of $G_2$-structures on such manifolds has led to numerous breakthroughs in $G_2$-geometry. This includes the constructions of complete metrics with holonomy $G_2$ in \cite{BS89,B13,FHN21}, the first non-trivial compact immortal $G_2$-Laplacian flow converging to a torsion-free $G_2$-structure in~\cite{HWY18}, and solitons of the $G_2$-Laplacian flow in~\cite{HN25,HJN25}. It is customary to call $G$ the symmetry group of $M$ and think of $G$-invariant tensors on $M$ as those that respect the symmetry. We assume in Theorem~B that the normal part of the form $\eta$ on the right-hand side of~\eqref{Poisson_equation} is nondegenerate.
This aligns with the nondegeneracy assumptions arising in previous works on~\eqref{Poisson_equation}, such as~\cite{BP25,H25}, and in the study of the prescribed Ricci curvature problem, as discussed below.


\begin{theorem_B}
Consider a manifold $M$ with a $G_2$-structure $\psi$ and a compact orientable embedded hypersurface $\Sigma$ in~$M$. Choose the background metric $g_0$ to coincide with the metric induced by~$\psi$. If problem~(\ref{Poisson_equation})--(\ref{initial_conditions}) has a solution in a neighbourhood of~$\Sigma$, then
\begin{equation}\label{suff_con}
\begin{split}
\langle\psi^+,\psi_{\scparal}\rangle &= -\langle\psi^-,\psi_{\scperp}\rangle,
\\
\int_\Sigma \langle\eta,\nu\wedge\psi_{\scperp}+\psi_{\scparal}\rangle -2\int_\Sigma\langle \eta_{\scperp},\psi_{\scperp}\rangle
&= \big\|d^\Sigma\psi_{\scparal}\big\|_{L^2}^2+\big\|\delta^\Sigma\psi_{\scparal}\big\|_{L^2}^2-\big\|d^\Sigma\psi_{\scperp}\big\|_{L^2}^2-\big\|\delta^\Sigma\psi_{\scperp}\big\|_{L^2}^2
\\
&\hphantom{=}~+ \big\|\psi^++d^\Sigma\psi_{\scperp}\big\|^2_{L^2} - 
\big\|\psi^- - \delta^\Sigma\psi_{\scparal}\big\|^2_{L^2}.
\end{split}
\end{equation}
If $M$ is cohomogeneity one with the symmetry group $G$ compact, connected and simple, $\Sigma$ is a principal orbit, the data $\eta$, $\psi$, $\psi^+$ and $\psi^-$ are $G$-invariant, and the inequality
\begin{align}\label{nondeg_assumtion}
\langle \eta_{\scperp},\psi_{\scperp}\rangle\ne0
\end{align}
and formulas~(\ref{suff_con}) hold, then $\Sigma$ has a neighbourhood in which problem~(\ref{Poisson_equation})--(\ref{initial_conditions}) possesses a unique $G$-invariant solution~$\phi$.
\end{theorem_B}

One can state an analogue of Theorem~B for \emph{closed} $G_2$-structures. We discuss this in Section~\ref{sec_closed}.

Theorem~B bears an intriguing resemblance to the existence and uniqueness results for the prescribed Ricci curvature equation obtained by the last-named author in earlier papers. Specifically, according to~\cite[Remark~2]{P13} and~\cite[Proposition~3.9]{P16}, the Gauss--Codazzi formula~\eqref{GaussCodazzi} yields the integrability condition for the problem
\begin{align*}
\Ric\gamma=T,\qquad \gamma^\Sigma=\hat\gamma,\qquad \II=\hat\II,
\end{align*}
in the cohomogeneity one setting. Here, the unknown is the Riemannian metric $\gamma$ on~$M$, the notation $\gamma^\Sigma$ stands for the restriction of $\gamma$ to~$\Sigma$, and $\hat\gamma$ and $\hat\II$ are a fixed metric and a fixed bilinear form on~$\Sigma$. The results of~\cite{P13,P16} rely on the nondegeneracy assumption
\begin{align*}
T(\mathbf n,\mathbf n)\ne0
\end{align*}
similar to~\eqref{nondeg_assumtion}. Several works, particularly those by DeTurck and his collaborators (see, e.g.,~\cite[Chapter~5]{B87} and~\cite{DTG99}), demonstrate that such assumptions are an essential part of the theory.

The resemblance between Theorem~B and the results of~\cite{P13,P16} underscores the connection, alluded to above, between our integral Gauss formula~\eqref{Gauss_formula} and the Gauss--Codazzi equation~\eqref{GaussCodazzi}. One more parallel between our analysis of~\eqref{Poisson_equation} and the study of Ricci curvature emerges from Anderson--Herzlich's paper~\cite{AH08} devoted to unique continuation properties of the Einstein equation. The uniqueness part of Theorem~B may be viewed as a $G_2$-geometric cohomogeneity one counterpart to their Proposition~3.7.

\section{Integral Gauss formula}

The purpose of this section is to prove Theorem A and Corollary~\ref{cor_ThmA}, thus relating the geometry of 
the manifold $M$ to the geometry of the hypersurface~$\Sigma$.

\subsection{Decompositions of first-order operators}

Let us lay some groundwork. We will perform our computations in a frame $\{X_i\}_{i=0}^{6}$ on a neighbourhood $U$ of an arbitrarily chosen point $p\in\Sigma$. We assume that $\{X_i\}_{i=0}^{6}$ is orthonormal with respect to the metric~$g$ and the restriction $X_0\vert_{U\cap \Sigma}$ equals~$\mathbf{n}$. The form field dual to $X_0$ on~$U$ coincides with $\nu$ on $U\cap\Sigma$. We preserve the notation $\nu$ for this form field.

Every $\varphi\in\Omega^k(U)$ appears as
\begin{equation*}
\varphi= \nu\wedge\varphi_N + \varphi_T,
\end{equation*}
where $\varphi_N\in\Omega^{k-1}(U)$ and $\varphi_T\in\Omega^k(U)$ are such that $\iota_{X_0}\,\varphi_N=\iota_{X_0}\,\varphi_T=0$. While $\varphi_N$ and $\varphi_T$ may depend on the choice of the frame, their restrictions to $U\cap\Sigma$ satisfy
\begin{align*}
\varphi_N|_{U\cap\Sigma} = \varphi_{\scperp},\qquad \varphi_T|_{U\cap\Sigma} = \varphi_{\scparal}.
\end{align*}
Let us split the exterior derivative $d$ and the codifferential $\delta$ into their ``normal" and ``tangential" components. Define 
\begin{equation*}
d^N\varphi = \nu\wedge(d\varphi_T)_N, \qquad  d^T = d-d^N, \qquad \delta^T\varphi =(-1)^k*^{-1}d^T*\varphi, \qquad \delta^N\varphi =(-1)^k*^{-1}d^N*\varphi,
\end{equation*}
where $*$ is the Hodge star on $M$ given by~$g$. Clearly,
\begin{align*}
    \delta=(-1)^k*^{-1}d\,* = (-1)^k*^{-1}d^T*+(-1)^k*^{-1}d^N* = \delta^T+\delta^N. 
\end{align*}
One can easily check that the Leibnitz rule holds for $d^N$ and~$d^T$. Next,
define $*^T$ to be the unique linear operator such that
\begin{equation*}
\sigma_T \wedge*^T\varphi_T = \langle\sigma_T,\varphi_T\rangle (\operatorname{vol}_M)_N, \qquad 
*^T (\nu\wedge \varphi_T)=0,
\end{equation*}
for all $\varphi,\sigma\in\Omega^k(U)$, where $\vol_M$ is the volume form of~$g$. The lemma below explains how $*^T$ relates to the Hodge star.

\begin{lemma}\label{relation between stars}
The following identities hold for the all $\varphi,\sigma\in \Omega^k(U)$:
\begin{enumerate}[label=(\roman*)]
\item $\sigma_T\wedge*^T\varphi_T = \varphi_T\wedge*^T \sigma_T = (-1)^k *^T\sigma_T\wedge\varphi_T$
\item $*\varphi_T = *^T\varphi_T\wedge \nu$
\item $*\,(\nu\wedge\varphi_T) = *^T\varphi_T$
\item $*\varphi = *^T\varphi_T\wedge \nu + *^T\varphi_N = (-1)^k\nu\wedge*^T\varphi_T + *^T\varphi_N$
\end{enumerate}
\end{lemma}

\begin{proof}
Identity \emph{(i)} is an immediate consequence of the definition of~$*^T$. To prove~\emph{(ii)}, observe that
\begin{align*}
\xi\wedge*^T\varphi_T\wedge \nu &= (\nu\wedge\xi_N+\xi_T)\wedge*^T\varphi_T \wedge \nu
\\
&= \xi_T\wedge*^T\varphi_T\wedge \nu = \langle\xi_T,\varphi_T\rangle(\operatorname{vol}_M)_N\wedge \nu = \langle\xi,\varphi_T\rangle\operatorname{vol}_M=\xi\wedge*\varphi_T
\end{align*}
for every $\xi\in\Omega^k(U)$.
Similarly,~\emph{(iii)} holds because
\begin{align*}
\xi\wedge*^T\varphi_T 
&= (\nu\wedge\xi_N+\xi_T)\wedge*^T\varphi_T 
= \nu\wedge\xi_N\wedge*^T\varphi_T 
\\
&= \nu\wedge\langle\xi_N,\varphi_T\rangle(\operatorname{vol}_M)_N= \langle \nu\wedge\xi_N, \nu\wedge\varphi_T\rangle\operatorname{vol}_M = \langle \xi, \nu\wedge\varphi_T\rangle\operatorname{vol}_M
\end{align*}
for $\xi\in\Omega^{k+1}(U)$. The second equality in this chain relies on the fact that $\iota_{X_0}\,\xi_T=\iota_{X_0}\,\varphi_T=0$, and the fourth uses the property $\langle\nu,\nu\rangle=1$. Finally, \emph{(iv)} is a consequence of \emph{(ii)} and~\emph{(iii)}.
\end{proof}

Our next lemma demonstrates that the operators $d^T,\delta^T$ and $*^T$ can be interpreted as extensions of $d^{\Sigma}$, $\delta^{\Sigma}$ and the Hodge star $*^{\Sigma}$ given by~$g^\Sigma$.

\begin{lemma}\label{Restriction of adapted frames}
For the all $\varphi\in \Omega^k(U)$,
\begin{align*}
(d^T\varphi_T)|_\Sigma = d^\Sigma\varphi_{\scparal},\qquad
(*^T\varphi_T)|_\Sigma = *^\Sigma\varphi_{\scparal},\qquad
(\delta^T\varphi_T)|_\Sigma = \delta^\Sigma\varphi_{\scparal}.
\end{align*}
\end{lemma}

\begin{remark}
Applying the result to $\varphi=\varphi_N$, we obtain the same identities with $\varphi_T$ and $\varphi_{\scparal}$ replaced by $\varphi_N$ and~$\varphi_{\scperp}$.
\end{remark}

\begin{proof}
Observe that
\begin{equation*}
d^\Sigma\varphi_{\scparal} = d^\Sigma(i^*\varphi) = i^*(d\varphi).
\end{equation*}
Since $\nu\wedge d^N\varphi=0$, the expression $i^*(d\varphi)$ equals~$(d^T\varphi_T)|_\Sigma$, which implies $(d^T\varphi_T)|_\Sigma = d^\Sigma\varphi_{\scparal}$. Next, take an arbitrary $\sigma\in\Omega^k(U\cap\Sigma)$. Extending this form to a $k$-form $\hat{\sigma}$ on $U$ in such a way that $\iota_{X_0}(\hat{\sigma})=0$, we compute
    \begin{equation*}
        \sigma\wedge(*^T\varphi_T)|_\Sigma = (\hat{\sigma}\wedge*^T\varphi_T)|_\Sigma = (\langle\hat{\sigma},\varphi_T\rangle(\operatorname{vol}_M)_N)|_\Sigma = \langle\sigma,\varphi_{\scparal}\rangle\operatorname{vol}_\Sigma=\sigma\wedge*^\Sigma\varphi_{\scparal}.
    \end{equation*}
Thus, $(*^T\varphi_T)|_\Sigma = *^\Sigma\varphi_{\scparal}$. Finally, we engage Lemma~\ref{relation between stars} to find 
\begin{align*}
\delta^T\varphi_T = (-1)^k*d^T*\varphi_T &= (-1)^k*d^T(*^T\varphi_T\wedge \nu)
\\
&= -*(\nu\wedge d^T*^T\varphi_T) = -*^Td^T*^T\varphi_T.
\end{align*}
The equality $(\delta^T\varphi_T)|_\Sigma = \delta^\Sigma\varphi_{\scparal}$ now follows from the formulas for $(d^T\varphi_T)|_\Sigma$ and $(*^T\varphi_T)|_\Sigma$.
\end{proof}

\subsection{Second-order operators}

In order to prove Theorem~A, we need to bring the Laplacian $\Delta^\Sigma$ into the picture. Clearly, this Laplacian is the restriction of the operator
\begin{align*}
\Delta^T=d^T\delta^T+\delta^Td^T.
\end{align*}

\begin{lemma}\label{lemma operators in adapted frame}
For every $\varphi\in\Omega^k(U)$,
\begin{equation}\label{(co)differential in adapted frame}
 \begin{split}
     d\varphi = (d^N\varphi_T - \nu\wedge d^T\varphi_N) + d^T\varphi_T,\qquad 
     \delta\varphi = -\nu\wedge\delta^T\varphi_N + (\delta^N(\nu\wedge\varphi_N)+\delta^T\varphi_T).
 \end{split}
 \end{equation}
The $G_2$-structure $\phi$ satisfies
\begin{equation}\label{Laplace in adapted frame}
\begin{split}
(\Delta_\phi\phi)_N &= 
\iota_{X_0}(d^N(\delta\phi)_T) + \Delta ^T\phi_N - \delta^T((d\phi)_N+d^T\phi_N),\\
(\Delta_\phi\phi)_T 
&= \delta^N(\nu\wedge (d\phi)_N) + \Delta^T\phi_T   + d^T((\delta\phi)_T-\delta^T\phi_T), 
\end{split}
\end{equation}
where $\Delta^T=d^T\delta^T+\delta^Td^T$.
\end{lemma}
\begin{proof}
Straightforward calculations yield
\begin{align*}
d\varphi&=(d^T+d^N)(\nu\wedge\varphi_N+\varphi_T) = (d^N\varphi_T-\nu\wedge d^T\varphi_N)+d^T\varphi_T,\\
\delta\varphi&=(\delta^T+\delta^N)(\nu\wedge\varphi_N+\varphi_T)=\delta^N(\nu\wedge\varphi_N) + \delta^T(\nu\wedge\varphi_N) + \delta^T\varphi_T.
\end{align*}
Using Lemma~\ref{relation between stars}, we obtain
\begin{align*}
    \delta^T(\nu\wedge\varphi_N) &= (-1)^k*d^T*(\nu\wedge\varphi_N) = (-1)^k * d^T*^T\varphi_N = (-1)^k (*^Td^T*^T\varphi_N)\wedge \nu\\
    &
    =\nu\wedge*\,(\nu\wedge(d^T*^T\varphi_N))= -\nu\wedge*\,d^T(\nu\wedge*^T\varphi_N) =(-1)^{k}\nu\wedge*\,d^T*\varphi_N = -\nu \wedge \delta^T\varphi_N,
\end{align*}
which proves ~\eqref{(co)differential in adapted frame}.
Similar calculations show that
    \begin{align*}
        d\delta\phi &= (d^T+d^N)(-\nu\wedge\delta^T\phi_N + (\delta\phi)_T) = \nu\wedge d^T\delta^T\phi_N + d^T(\delta\phi)_T+ d^N(\delta\phi)_T,
        \\
        \delta d\phi &= (\delta^T+\delta^N)(\nu\wedge (d\phi)_N + d^T\phi_T) = \delta^N(\nu\wedge(d\phi)_N) + \delta^T(\nu\wedge(d\phi)_N) + \delta^Td^T\phi_T\\
        &= \delta^N(\nu\wedge(d\phi)_N) - (\nu\wedge\delta^T(d\phi)_N) + \delta^Td^T\phi_T.
    \end{align*}
    After extracting the $N$ and $T$ components, we are left with
    \begin{align*}
        (\Delta_\phi\phi)_N &= d^T\delta^T\phi_N + \iota_{X_0}(d^N(\delta\phi)_T) - \delta^T(d\phi)_N,\\
        (\Delta_\phi\phi)_T &= d^T(\delta\phi)_T + \delta^N(\nu\wedge (d\phi)_N) + \delta^Td^T\phi_T.
    \end{align*}
Adding $\delta^Td^T\phi_N$ to the first equation and $d^T\delta^T\phi_T$ to the second produces~\eqref{Laplace in adapted frame}.
\end{proof}

\begin{remark}\label{rem_2nd_order_closed}
By formula~\eqref{(co)differential in adapted frame}, if $d\phi=0$, then $d^\Sigma\phi_{\scparal}=0$. In this case, the second line in~\eqref{Laplace in adapted frame} implies
\begin{equation*}
(\Delta_\phi\phi)_{\scparal} = d^\Sigma(\delta\phi)_{\scparal}.
\end{equation*}
\end{remark}

The following proposition establishes formula~\eqref{first order condition}.

\begin{proposition}\label{prop_1st_order_cond}
The $G_2$-structure $\phi$ satisfies
\begin{equation*}
(d\phi)_{\scperp}\wedge(*\phi)_{\scperp} = (\delta\phi)_{\scparal}\wedge (*\phi)_{\scparal}
\end{equation*}
on the hypersurface $\Sigma$.
\end{proposition}

\begin{proof}
As shown in~\cite[page~553]{BR87},
\begin{equation*}
(*d\phi)\wedge\phi = (\delta\phi)\wedge*\phi.
\end{equation*}
Taking the wedge product with~$\nu$, we obtain
\begin{equation}\label{normal wedge Bryant}
\nu\wedge (*d\phi)_T\wedge\phi_T = \nu\wedge(\delta\phi)_T\wedge(*\phi)_T.
\end{equation}
Formula~$(iv)$ of Lemma~\ref{relation between stars} implies
\begin{align*}
\nu\wedge (*d\phi)_T\wedge\phi_T &= \nu\wedge*^T(d\phi)_N\wedge\phi_T = -*(d\phi)_N\wedge\phi_T \\ 
&= -(d\phi)_N\wedge*\phi_T = -(d\phi)_N\wedge*^T\phi_T \wedge \nu = \nu\wedge(d\phi)_N\wedge(*\phi)_N.
\end{align*}
Substituting this into~\eqref{normal wedge Bryant} and taking the interior product with $\nu$ completes the proof.
\end{proof}

The following proposition combines the two formulas in~\eqref{Laplace in adapted frame}. This give us the pointwise variant of Theorem~A.

\begin{proposition}\label{prop_pwise_Gauss}
The $G_2$-structure $\phi$ satisfies
\begin{equation*}
\begin{split}
(\Delta_\phi\phi)_{\scperp}\wedge(*\phi)_{\scparal} &+ (\Delta_\phi\phi)_{\scparal} \wedge(*\phi)_{\scperp}
\\
&= \Delta^\Sigma\phi_{\scperp}\wedge(*\phi)_{\scparal} +\Delta^\Sigma\phi_{\scparal}\wedge(*\phi)_{\scperp}
\\
&\hphantom{=}~- ((d\phi)_{\scperp}+d^\Sigma\phi_{\scperp}) \wedge *^\Sigma(d\phi)_{\scperp} - \delta^\Sigma((d\phi)_{\scperp}+d^\Sigma\phi_{\scperp})\wedge(*\phi)_{\scparal}
\\
&\hphantom{=}~
+ ((\delta\phi)_{\scparal}-\delta^\Sigma\phi_{\scparal})\wedge*^\Sigma(\delta\phi)_{\scparal} + d^\Sigma((\delta\phi)_{\scparal}-\delta^\Sigma\phi_{\scparal})\wedge(*\phi)_{\scperp}
\end{split}
\end{equation*}
on the hypersurface $\Sigma$.
\end{proposition}

\begin{proof}
Formulas~\eqref{Laplace in adapted frame} imply
\begin{align}\label{NTsumwedge}
\begin{split}
(\Delta_\phi\phi)_N &\wedge(*\phi)_T + (\Delta_\phi\phi)_T \wedge(*\phi)_N 
\\
&= 
\iota_{X_0}(d^N(\delta\phi)_T)\wedge(*\phi)_T + \Delta ^T\phi_N \wedge(*\phi)_T - \delta^T((d\phi)_N+d^T\phi_N) \wedge(*\phi)_T
\\
&\hphantom{=}~+\delta^N(\nu\wedge (d\phi)_N) \wedge(*\phi)_N + \Delta^T\phi_T \wedge(*\phi)_N + d^T((\delta\phi)_T-\delta^T\phi_T) \wedge(*\phi)_N.
\end{split}
\end{align}
By Proposition~\ref{prop_1st_order_cond} and Lemma~\ref{relation between stars},
\begin{equation*}
(\delta\phi)_T\wedge (*\phi)_T = -(d\phi)_N\wedge *^T \phi_T = *^T(d\phi)_N\wedge\phi_T.
\end{equation*}
Applying $d^N$, we find
\begin{equation}\label{d^N applied - prop pwise Gauss}
d^N(\delta\phi)_T\wedge(*\phi)_T + (\delta\phi)_T\wedge d^N(*\phi)_T = d^N(*^T (d\phi)_N)\wedge \phi_T - *^T(d\phi)_N\wedge d^N\phi_T.
\end{equation}
In light of Lemma~\ref{relation between stars} and~\eqref{(co)differential in adapted frame}, the second term on the left transforms as follows:
\begin{align*}
(\delta\phi)_T\wedge d^N(*\phi)_T
&= (\delta\phi)_T\wedge d^N(*^T\phi_N) = (\delta\phi)_T\wedge d^N*(\nu\wedge\phi_N) 
= -(\delta\phi)_T\wedge *\,\delta^N(\nu\wedge\phi_N)
\\
&=-(\delta\phi)_T\wedge*\,((\delta\phi)_T-\delta^T\phi_T) = -(\delta\phi)_T\wedge *^T((\delta\phi)_T-\delta^T\phi_T)\wedge \nu.
\end{align*}
Similarly,
\begin{align*}
d^N(*^T (d\phi)_N)\wedge \phi_T
&= d^N(*(\nu\wedge(d\phi)_N))\wedge\phi_T = \delta^N(\nu\wedge (d\phi)_N)\wedge*\phi_T 
\\
&
= \delta^N(\nu\wedge (d\phi)_N)\wedge*^T\phi_T\wedge \nu = -\delta^N(\nu\wedge(d\phi)_N)\wedge(*\phi)_N\wedge \nu.
\end{align*}
Another invocation of~\eqref{(co)differential in adapted frame} gives us
\begin{align*}
*^T(d\phi)_N\wedge d^N\phi_T 
= -*^T(d\phi)_N\wedge ((d\phi)_N+d^T\phi_N)\wedge \nu.
\end{align*}
Substituting the obtained expressions into~\eqref{d^N applied - prop pwise Gauss}, taking interior product with~$X_0$, and reordering, we find
\begin{align*}
\begin{split}
\iota_{X_0}(d^N(\delta\phi)_T)\wedge(*\phi)_T &+\delta^N(\nu\wedge(d\phi)_N)\wedge(*\phi)_N 
\\
&= (\delta\phi)_T\wedge *^\Sigma((\delta\phi)_T-\delta^T\phi_T) + *^\Sigma(d\phi)_N\wedge ((d\phi)_N 
+d^T\phi_N).
\end{split}
\end{align*}
Combined with~\eqref{NTsumwedge}, this identity implies the result.
\end{proof}

We integrate the formula in Proposition~\ref{prop_pwise_Gauss} to recover the assertion of Theorem~A.

\begin{proof}[Proof of Theorem~A]
Since $\langle\nu,\nu\rangle=1$ and $\langle\varphi_{\scparal},\nu\wedge\sigma_{\scperp}\rangle=0$ for all $\varphi,\sigma\in\Omega^k(M)$,
\begin{equation}\label{combining the full Laplace from parts}
\begin{split}
\langle(\Delta_\phi\phi)_{\scperp},\phi_{\scperp}\rangle+ \langle(\Delta_\phi\phi)_{\scparal},\phi_{\scparal}\rangle 
&= \langle\nu\wedge(\Delta_\phi\phi)_{\scperp},\nu\wedge\phi_{\scperp}\rangle + \langle(\Delta_\phi\phi)_{\scparal},\phi_{\scparal}\rangle
\\
&
= \langle\nu\wedge(\Delta_\phi\phi)_{\scperp},\phi\rangle + \langle(\Delta_\phi\phi)_{\scparal},\phi\rangle =\langle\Delta_\phi\phi,\phi\rangle.
\end{split}
\end{equation}
Item~\emph{(iv)} of Lemma~\ref{relation between stars} enables us to re-state the formula in Proposition~\ref{prop_pwise_Gauss} as
\begin{align*}
(\Delta_\phi\phi)_{\scperp}\wedge*^\Sigma\phi_{\scperp} &- (\Delta_\phi\phi)_{\scparal} \wedge*^\Sigma\phi_{\scparal}
\\
&= \Delta^\Sigma\phi_{\scperp}\wedge*^\Sigma\phi_{\scperp} -\Delta^\Sigma\phi_{\scparal}\wedge*^\Sigma\phi_{\scparal}
\\
&\hphantom{=}~- ((d\phi)_{\scperp}+d^\Sigma\phi_{\scperp}) \wedge *^\Sigma(d\phi)_{\scperp} - \delta^\Sigma((d\phi)_{\scperp}+d^\Sigma\phi_{\scperp})\wedge*^\Sigma\phi_{\scperp}
\\
&\hphantom{=}~+ ((\delta\phi)_{\scparal}-\delta^\Sigma\phi_{\scparal})\wedge*^\Sigma(\delta\phi)_{\scparal} - d^\Sigma((\delta\phi)_{\scparal}-\delta^\Sigma\phi_{\scparal})\wedge*^\Sigma\phi_{\scparal}.
\end{align*}
Equivalently,
\begin{align*}
\langle(\Delta_\phi\phi)_{\scperp},\phi_{\scperp}\rangle&- \langle(\Delta_\phi\phi)_{\scparal},\phi_{\scparal}\rangle
\\
&=\langle\Delta^\Sigma\phi_{\scperp},\phi_{\scperp}\rangle -\langle\Delta^\Sigma\phi_{\scparal},\phi_{\scparal}\rangle
\\
&\hphantom{=}~- \langle (d\phi)_{\scperp}+d^\Sigma\phi_{\scperp},(d\phi)_{\scperp}\rangle - \langle\delta^\Sigma((d\phi)_{\scperp}+d^\Sigma\phi_{\scperp}),\phi_{\scperp}\rangle
\\
&\hphantom{=}~+ \langle(\delta\phi)_{\scparal}-\delta^\Sigma\phi_{\scparal},(\delta\phi)_{\scparal}\rangle - \langle d^\Sigma((\delta\phi)_{\scparal}-\delta^\Sigma\phi_{\scparal}),\phi_{\scparal}\rangle.
\end{align*}
We integrate this identity over $\Sigma$ to find
\begin{align*}
\int_\Sigma(\langle(\Delta_\phi\phi)_{\scperp},\phi_{\scperp}\rangle &- \langle(\Delta_\phi\phi)_{\scparal},\phi_{\scparal}\rangle)
\\
&= \int_\Sigma(\langle\Delta^\Sigma\phi_{\scperp},\phi_{\scperp}\rangle -\langle\Delta^\Sigma\phi_{\scparal},\phi_{\scparal}\rangle)
\\
&\hphantom{=}~- \int_\Sigma(\langle(d\phi)_{\scperp}+d^\Sigma\phi_{\scperp},(d\phi)_{\scperp}\rangle + \langle(d\phi)_{\scperp}+d^\Sigma\phi_{\scperp},d^\Sigma\phi_{\scperp}\rangle)
\\
&\hphantom{=}~+\int_\Sigma(\langle(\delta\phi)_{\scparal}-\delta^\Sigma\phi_{\scparal},(\delta\phi)_{\scparal}\rangle - \langle (\delta\phi)_{\scparal}-\delta^\Sigma\phi_{\scparal},\delta^\Sigma\phi_{\scparal}\rangle).
\end{align*}
It remains to add $2\int_\Sigma\langle(\Delta_\phi\phi)_{\scperp},\phi_{\scperp}\rangle$ to both sides and invoke~\eqref{combining the full Laplace from parts}.
\end{proof}

\begin{proof}[Proof of Corollary~\ref{cor_ThmA}]
The first formula follows from the definition of the Hodge Laplacian and the adjointness of $d^\Sigma$ and~$\delta^\Sigma$.
The second one is a consequence of the first and the observation that $d^\Sigma\phi_{\scparal}=0$ whenever $d\phi=0$ (see~\eqref{(co)differential in adapted frame}).
\end{proof}

\section{Poisson equation}\label{sec_Poisson}

The first assertion of Theorem~B is an immediate consequence of Proposition~\ref{prop_1st_order_cond} and Theorem~A. The purpose of this section is to prove the second assertion. From now on, let $M$ be a cohomogeneity one manifold with compact, connected and simple symmetry group $G$. Assume that the $G_2$-structure~$\psi$, and hence the induced metric~$g_0$, is $G$-invariant. Choose a $g_0$-geodesic $\gamma:I\to M$, parametrised by arc-length, that starts on the principal orbit~$\Sigma$ and runs $g_0$-orthogonally to the other principal orbits. As is standard, we use $\gamma$ to identify a tubular neighborhood of $\Sigma$ in $M$ with the product $I\times G/K$, where $I$ is an interval containing~0 and $K$ is a closed subgroup of~$G$. Upon this identification,
\begin{align*}
\Sigma=\{0\}\times G/K.
\end{align*}
If $r$ is the arc-length parameter of~$\gamma$, then $\nu=dr$ on~$\Sigma$.

Denote by $\mathfrak{g}$ and $\mathfrak{k}$ the Lie algebras of $G$ and~$K$. Let $\mathfrak m$ be the orthogonal complement of $\mathfrak k$ in $\mathfrak g$ with respect to the Killing form of~$\mathfrak g$. As is standard, we identify $\mathfrak m$ with the tangent space $T_K(G/K)$ at the coset~$K$. The isotropy representation of $G/K$ is isomorphic to the representation $\Ad^K:K\to GL(\mathfrak m)$ induced by the adjoint representation $\Ad:G\to GL(\mathfrak g)$. Since $M$ supports a $G$-invariant $G_2$-structure, $G/K$ must have a $G$-invariant $SU(3)$-structure; see~\cite{CS02}. Choosing a suitable basis of~$\mathfrak m$ and interpreting $SU(3)$ as a subgroup of $GL(6)$ under the natural embedding, we can identify $GL(\mathfrak m)$ with $GL(6)$ and the image $\Ad^K(K)$ with a subgroup of~$SU(3)$.
In what follows, we denote $\tilde K=\Ad^K(K)$ and view $\mathfrak m$ as a $\tilde K$-module via the fundamental representation of $SU(3)$ unless indicated otherwise. According to~\cite[Theorem~3.1]{CS02}, the principal orbit type $G/K$ must be, up to a finite quotient, one of the following:
\begin{itemize}
\item
The sphere $\mathbb S^6=G_2/SU(3)$
\item
The flag manifold~$SU(3)/\mathbb T^2$
\item
The projective space $\mathbb CP^3=Sp(2)/SU(2)U(1)$
\end{itemize}
The lack of uniform description of $G$-invariant differential forms on these spaces forces us to consider each case separately. Calculating the $G_2$-Laplacian, we reduce the Poisson equation~\eqref{Poisson_equation} to a ``non-elliptic" system of ODEs. A variant of DeTurck's trick helps us prove the existence of solutions, and an argument inspired by the analysis of the Ricci flow yields uniqueness; cf.~\cite{P13}.

We recall a series of useful facts from homogeneous geometry and representation theory in Section~\ref{subsec_rep_theory}. The second assertion of Theorem~B is proven for each orbit type in Sections~\ref{subsec_spherical}, \ref{subsec_flag_mf} and~\ref{subsec_proj_orb}.

\subsection{Invariant forms on homogeneous spaces with $SU(3)$-structures}\label{subsec_rep_theory}

The space of $G$-invariant differential forms for each orbit type is described in~\cite{CS02}. We repeat these descriptions, slightly expanded, throughout the present paper for the convenience of the reader. In doing so, we will rely on the discussion of differential forms on spaces with $SU(3)$-structures in~\cite{FSS94,ChS02}.

Denote by $S^2(\mathfrak m^*)$ and $\Lambda^k(\mathfrak m^*)$ the second symmetric tensor power and the $k$th exterior power of~$\mathfrak m^*$. The identification between $\mathfrak m$ and $T_K(G/K)$ yields an identification between the space of $G$-invariant Riemannian metrics on~$G/K$ and the space $S_+^2(\mathfrak m^*)^K$ of $\tilde K$-invariant scalar products on~$\mathfrak m$. Similarly, we can identify the space of $G$-invariant $k$-forms on $G/K$ with the space $\Lambda^k(\mathfrak{m}^*)^K$ of $\tilde K$-invariant $k$-forms on~$\mathfrak{m}$.

Given $r\in(-1,1)$, 
the $G_2$-structure $\psi$ 
induces an $SU(3)$-structure on the orbit~$\{r\}\times G/K$, and hence an almost complex structure~$J$ on~$\mathfrak m$. Suppose that $\Lambda^{p,q}$ is the space of complex-valued forms on $\mathfrak m$ of type~$(p,q)$ with respect to~$J$. Let $\Lambda_0^{p,q}$ be the primitive part of~$\Lambda^{p,q}$. If $V$ is a complex vector space with a real structure, $[[V]]$ and $[V]$ stand for the realification and the real part of~$V$, respectively. The fundamental representation of $SU(3)$ gives rise to representations of $\tilde K$ on $S^2(\mathfrak m^*)$, $\Lambda^k(\mathfrak m^*)$, $[[\Lambda^{p,q}]]$, $[[\Lambda_0^{p,q}]]$ and~$[\Lambda_0^{p,p}]$.
In what follows, we view $S^2(\mathfrak m^*)$, $\Lambda^k(\mathfrak m^*)$, $[[\Lambda^{p,q}]]$, $[[\Lambda_0^{p,q}]]$ and~$[\Lambda_0^{p,p}]$ as $\tilde K$-modules via these representations.
By duality,
\begin{align}\label{eq_Lambda1_m}
\Lambda^1(\mathfrak m^*)\cong\mathfrak m.
\end{align}
The discussion in~\cite{FSS94,ChS02} demonstrates that
\begin{align}\label{decomposition of exterior powers under su(3)}
\begin{split}
\Lambda^1(\mathfrak{m}^*) &\cong [[\Lambda^{1,0}]],
\\
\Lambda^2(\mathfrak{m}^*) &\cong [[\Lambda^{1,0}]] \oplus [\Lambda^{1,1}_0] \oplus \mathbb{R},
\\
\Lambda^3(\mathfrak{m}^*) &\cong \mathbb{R} \oplus \mathbb{R} \oplus [[\Lambda^{2,1}_0]]\oplus [[\Lambda^{1,0}]].
\end{split}
\end{align}
The 1-dimensional $\mathbb R$ summands are interpreted here as $\tilde K$-modules via the trivial representation, and $[\Lambda_0^{1,1}]$ is isomorphic to the Lie algebra $\mathfrak{su}(3)$ viewed as a $\tilde K$-module via the adjoint representation of~$SU(3)$. Arguing as in~\cite{FSS94}, one easily shows that
\begin{align}\label{decomp_tens}
S^2(\mathfrak{m}^*) &\cong \mathbb{R} \oplus [[\Lambda^{2,1}_0]] \oplus [\Lambda^{1,1}_0].
\end{align}
The space $\Lambda^k(\mathfrak m^*)^K$ coincides with the largest trivial submodule of $\Lambda^k(\mathfrak m^*)$. Equality~\eqref{decomp_tens} will help us compute the dimension of~$\Lambda^3(\mathfrak m^*)^K$ in Sections~\ref{subsubsection - invariant forms and G2 Laplacian} and~\ref{subsubsec_inv_orb_red}.

\subsection{Spherical orbits}\label{subsec_spherical}

Assume that $G=G_2$ and $K=SU(3)$ with the standard embedding given by the realisation of $G_2$ as a subgroup of~$SO(7)$; see~\cite[p.~539]{BR87}. In this case, the homogeneous space $G/K$ is isotropy irreducible.

\subsubsection{Invariant forms on the orbit}\label{subsubsection - invariant forms and G2 Laplacian}

The representation $\Ad^K$ is isomorphic to the fundamental representation of $SU(3)$ on~$\mathbb{C}^3$; see~\cite{CS02}. Fix a basis $\{e_i\}_{i=1}^{6}$ of $\mathfrak{m}$ that maps to the standard basis of $\mathbb C^3$ (viewed as $\mathbb R^6$) under this isomorphism, and denote the corresponding dual basis~$\{e^i\}_{i=1}^{6}$. Engaging Schur's lemma and using the inclusion~$SU(3)<SO(6)$, one can prove that every $G$-invariant metric on $G/K$ is a scalar multiple of the metric
\begin{align*}
g_1=\sum_{i=1}^{6} e^i\otimes e^i.
\end{align*}
For every $r\in I$, the space $\Omega^k(\{r\}\times G/K)^G$ of $G$-invariant $k$-forms on the orbit $\{r\}\times G/K$ is isomorphic to~$\Lambda^k(\mathfrak{m}^*)^K$. Formula~\eqref{eq_Lambda1_m} and the irreducibility of $\mathfrak m$ show that $\Lambda^1(\mathfrak{m}^*)^K$ contains nothing but zero. We conclude that
$$\Omega^1(\{r\}\times G/K)^G=\{0\}.$$
Next, the second line in~\eqref{decomposition of exterior powers under su(3)} implies
\begin{align*}
\begin{split}
\Lambda^2(\mathfrak{m}^*) &\cong\mathfrak m\oplus[\Lambda_0^{1,1}]\oplus\mathbb R.
\end{split}
\end{align*}
The $\tilde K$-module $[\Lambda_0^{1,1}]$ is irreducible because it is isomorphic to $\mathfrak{su}(3)$ as discussed in Section~\ref{subsec_rep_theory}. Consequently, $\Lambda^2(\mathfrak{m}^*)^K$ is 1-dimensional. One checks easily that the 2-form
\begin{equation*}
\omega = e^{12}+e^{34}+e^{56}
\end{equation*}
is $\tilde K$-invariant (here and in what follows, $e^{i_1\cdots i_l}=e^{i_1}\wedge \cdots \wedge e^{i_l}$). Thus, 
$$\Omega^2(\{r\}\times G/K)^G=\text{span}\{\omega\}.$$
Finally, the third line of~\eqref{decomposition of exterior powers under su(3)} yields
\begin{align*}
\Lambda^3(\mathfrak{m}^*) &\cong \mathbb{R} \oplus \mathbb{R} \oplus [[\Lambda^{2,1}_0]]\oplus \mathfrak m.
\end{align*}
Since $S^2(\mathfrak m^*)$ contains only one scalar product up to scaling, formula~\eqref{decomp_tens} shows that $[[\Lambda^{2,1}_0]]$ has no trivial submodules. This enables us to conclude that $\Lambda^3(\mathfrak{m}^*)^K$ is 2-dimensional. It is straightforward to verify the $\tilde K$-invariance of the 3-forms
\begin{equation}\label{invariant tensors in coordinates - irreducible - alphabeta}
\alpha = e^{246}-e^{235}-e^{145}-e^{136}, \qquad  \beta = e^{135}-e^{146}-e^{236}-e^{245}.
\end{equation}
Therefore,
$$\Omega^3(\{r\}\times G/K)^G=\text{span}\{\alpha,\beta\}.$$
Let $*_1$ and $\vol_1$ denote the Hodge star and the volume form $e^{123456}$ associated with~$g_1$. The following identities hold for $\omega$, $\alpha$ and~$\beta$:
\begin{align}\label{exterior derivative for irreducible case}
    d\omega=3\alpha, \qquad d\beta = -2\omega^2, \qquad \omega\wedge*_1\omega=\frac{1}{2}\omega^3=3\vol_1, \qquad \alpha\wedge*_1\alpha=\alpha\wedge\beta = 4\vol_1;
\end{align}
see, e.g.,~\cite[Section~2]{J22} and~\cite{CS02}.

\begin{remark}
The metric $g_1$ is a round metric on the 6-dimensional sphere, $\omega$ is a K\"ahler form as in~\cite{C58}, and $\alpha$ and $\beta$ are the real and the imaginary part of a complex volume form. 
\end{remark}

\subsubsection{Invariant forms on $M$ and the $G_2$-Laplacian}

The discussion in Section~\ref{subsubsection - invariant forms and G2 Laplacian}
implies that $dr$ and $\omega$ are the only $G$-invariant one-form and two-form on~$M$ up to multiplication by a scalar coefficient, possibly depending on~$r$. Every $G$-invariant $3$-form on $M$ appears as
\begin{equation}\label{3formisotropyirrep}
\phi=h\,dr\wedge\omega + f_\alpha\,\alpha + f_\beta\,\beta,
\end{equation}
where $h$, $f_\alpha$ and $f_\beta$ are smooth scalar functions on $(-1,1)$. Furthermore, $\phi$ is a $G_2$-structure if and only if $h\ne0$ and
\begin{align*}
\rho=\sqrt{f_\alpha^2+f_\beta^2}\ne0
\end{align*}
on $(-1,1)$. In this case, the metric induced by $\phi$ is
\begin{equation*}
g=\rho^{-\frac{4}{3}}h^2 dr\otimes dr+ \rho^\frac{2}{3}g_1;
\end{equation*}
see~\cite[Section~5.2]{H25}. It is easy to verify that
\begin{equation}\label{Hodge stars for induced metric - irreducible}
\begin{aligned}
    * \omega &= \frac{1}{2}h\,dr\wedge\omega^2, \quad * (dr\wedge\omega) = \frac{1}{2}\rho^{\frac{4}{3}}h^{-1}\omega^2, \quad * \alpha= -\rho^{-\frac{2}{3}}h\,dr\wedge\beta, \quad * \beta = \rho^{-\frac{2}{3}}h\,dr\wedge\alpha,
\end{aligned}
\end{equation}
where $*$ is the Hodge star given by~$g$, and
\begin{align}\label{SigmaHodge - irreducible}
   *_0\omega = \frac{1}{2}\rho^{\frac{2}{3}}(0)\,\omega^2, \qquad *_0\alpha=\beta, \qquad *_0\beta=-\alpha, \qquad *_01=\rho(0)^2\vol_1, 
\end{align}
where $*_0$ is the Hodge star given by~$g_0^{\Sigma}=i^*g$.
These relations enable the following computations.

\begin{lemma}\label{lem_starddelta_irred}
If $\phi$ is a $G_2$-structure on $M$ given by~(\ref{3formisotropyirrep}), then
    \begin{align*}
        *\phi &= \frac{1}{2}\rho^{\frac{4}{3}}\omega^2 - \rho^{-\frac{2}{3}}f_\alpha h\,dr\wedge\beta + \rho^{-\frac{2}{3}}f_\beta h\,dr\wedge\alpha,
        \\
        d\phi &= (f_\alpha'-3h)\,dr\wedge\alpha + f_\beta'\,dr\wedge\beta -2f_\beta\,\omega^2, 
        \qquad\qquad 
        \delta\phi = -\frac{4}{3}\rho^{\frac{1}{3}}h^{-1}(\rho'-3\rho^{-1}f_\alpha h)\,\omega.
\end{align*}
\end{lemma}

\begin{proof}
The formulas for $*\phi$ and $d\phi$ are direct consequences of~\eqref{Hodge stars for induced metric - irreducible} and~\eqref{exterior derivative for irreducible case}, respectively.
Next,
\begin{align*}
\delta \phi &= -* d* (h\,dr\wedge\omega + f_\alpha\alpha + f_\beta\beta)
= -* d\Big(\frac{1}{2}\rho^{\frac{4}{3}}\omega^2 - \rho^{-\frac{2}{3}}f_\alpha h\,dr\wedge\beta + \rho^{-\frac{2}{3}}f_\beta h\,dr\wedge\alpha\Big)
\\
&= -*\Big(\Big(\frac{2}{3}\rho^\frac{1}{3}\rho' - 2\rho^{-\frac{2}{3}}f_\alpha h\Big) \,dr\wedge\omega^2\Big)
= -\frac{4}{3}\rho^{\frac{1}{3}}h^{-1}(\rho'-3\rho^{-1}f_\alpha h)\,\omega.
\end{align*}
\end{proof}

\begin{remark}\label{rem_closed_irred}
If $\phi$ is closed, then
\begin{align*}
f_\alpha'=3h, \qquad f_\beta'=f_\beta=0.
\end{align*}
Consequently, $\rho=f_\alpha$ and
\begin{equation*}
\delta\phi =-\frac{4}{3}(\rho^{\frac{1}{3}}\rho'h^{-1} - 3f_\alpha\rho^{-\frac{2}{3}})\,\omega =0.
\end{equation*}
In other words, every closed $G_2$-structure is coclosed. This implies that the local solutions to the Poisson equation~\eqref{Poisson_equation} obtained on $(0,1)\times G/K$ for closed $G_2$-invariant $\eta$ via~\cite[Theorem~1.1]{BP25} are never $G_2$-invariant.
\end{remark}

We can now compute the $G_2$-Laplacian.

\begin{proposition}\label{G2 Laplacian - irreducible case}
If $\phi$ is a $G_2$-structure on $M$ given by~(\ref{3formisotropyirrep}), then
\begin{align*}
\Delta_\phi\phi &= -\Big(\Big(\frac{4}{3}\rho^{\frac{1}{3}}h^{-1}(\rho'-3\rho^{-1}f_\alpha h)\Big)'+4\rho^{-\frac{2}{3}}(f_\alpha'-3h)\Big)\,dr\wedge\omega
\\
&\hphantom{=}~- (4\rho^{\frac{1}{3}}h^{-1}(\rho'-3\rho^{-1}f_\alpha h)+\rho^\frac{2}{3}h^{-1}(\rho^{\frac{2}{3}}h^{-1}(f_\alpha'-3h))')\,\alpha
\\
&\hphantom{=}~+ \rho^\frac{2}{3}h^{-1}(12\rho^{-\frac{4}{3}}f_\beta h -(\rho^{\frac{2}{3}}h^{-1}f_\beta')')\,\beta.
\end{align*}
\end{proposition}

\begin{proof}
Lemma~\ref{lem_starddelta_irred} and~\eqref{exterior derivative for irreducible case} imply
\begin{align*}
d\delta\phi
= -\Big(\frac{4}{3}\rho^{\frac{1}{3}}h^{-1}(\rho'-3\rho^{-1}f_\alpha h)\Big)'dr\wedge\omega - 4h^{-1}\rho^{\frac{1}{3}}(\rho'-3\rho^{-1}f_\alpha h)\,\alpha.
\end{align*}
Similarly,
\begin{align*}
\delta d\phi &= *\,d*((-3h+f_\alpha')\,dr\wedge\alpha + f_\beta'\,dr\wedge\beta -2f_\beta\omega^2) \\
&= *\,d(\rho^{\frac{2}{3}}h^{-1}(f_\alpha'-3h)\beta - \rho^{\frac{2}{3}}h^{-1}f_\beta'\alpha - 4\rho^{-\frac{4}{3}}f_\beta h\,dr\wedge\omega  ) \\
&= * ((\rho^{\frac{2}{3}}h^{-1}(f_\alpha'-3h))'dr\wedge\beta -2\rho^{\frac{2}{3}}h^{-1}(f_\alpha'-3h)\,\omega^2 + (-(\rho^{\frac{2}{3}}h^{-1}f_\beta')' + 12\rho^{-\frac{4}{3}}f_\beta h )\,dr\wedge\alpha) \\
&= -\rho^\frac{2}{3}h^{-1}(\rho^{\frac{2}{3}}h^{-1}(f_\alpha'-3h))'\alpha -4\rho^{-\frac{2}{3}}(f_\alpha'-3h)\,dr\wedge\omega + \rho^\frac{2}{3}h^{-1}(12\rho^{-\frac{4}{3}}f_\beta h -(\rho^{\frac{2}{3}}h^{-1}f_\beta')')\,\beta.
\end{align*}
\end{proof}

\subsubsection{Existence of solutions}\label{subsubsec_exist_irred}

The 3-form $\eta$ on the right-hand side of the Poisson equation~\eqref{Poisson_equation} appears as
\begin{align*}
\eta=\eta_r\,dr\wedge\omega +\eta_\alpha\,\alpha+\eta_\beta\,\beta,
\end{align*}
where $\eta_r$, $\eta_\alpha$ and $\eta_\beta$ are functions on~$(-1,1)$. Assumption~\eqref{nondeg_assumtion} means that $\eta_r(0)\ne0$. Proposition~\ref{G2 Laplacian - irreducible case} reduces~\eqref{Poisson_equation} to the system
\begin{align}\label{sys_raw_Poisson_irred}
\begin{split}
    -\frac43(\rho^{\frac13}h^{-1}(\rho'-3\rho^{-1}f_\alpha h))'-4\rho^{-\frac23}(f_\alpha'-3h)&=\eta_r,
\\
-4\rho^{\frac13}h^{-1}(\rho'-3\rho^{-1}f_\alpha h) -\rho^{\frac23}h^{-1}(\rho^{\frac23}h^{-1}(f_\alpha'-3h))'&=\eta_\alpha,
\\
\rho^{\frac23}h^{-1}(12\rho^{-\frac43}f_\beta h-(\rho^{\frac23}h^{-1}f_\beta')')&=\eta_\beta,
\end{split}
\end{align}
with 
unknowns $h$, $f_\alpha$ and~$f_\beta$.
We open the brackets in the last two equations, rearrange the terms and note that 
\begin{align}\label{eq_rho_prime}
\rho'&=\rho^{-1}(f_\alpha f_\alpha'+f_\beta f_\beta').
\end{align}
This yields
\begin{align}\label{sys_ab_Poisson_irred}
\begin{split}
f_\alpha''&=-\rho^{-\frac43}h^2\eta_\alpha+12\rho^{-2}f_\alpha h^2-\frac23\rho^{-2}(f_\alpha'+3h)(f_\alpha f_\alpha'+f_\beta f_\beta')+f_\alpha'h^{-1}h',
\\
f_\beta''&=-\rho^{-\frac43}h^2\eta_\beta+12\rho^{-2}f_\beta h^2-\frac23\rho^{-2}f_\beta'(f_\alpha f_\alpha'+f_\beta f_\beta')+f_\beta'h^{-1}h'.
\end{split}
\end{align}
Consequently,
\begin{align*}
\rho''
&=-\rho^{-2}\rho'f_\alpha f_\alpha'+\rho^{-1}f_\alpha'^2+\rho^{-1}f_\alpha f_\alpha''-\rho^{-2}\rho'f_\beta f_\beta'+\rho^{-1}f_\beta'^2+\rho^{-1}f_\beta f_\beta'' 
\\
&=-\rho^{-3}(f_\alpha f_\alpha'+f_\beta f_\beta')^2+\rho^{-1}(f_\alpha'^2+f_\beta'^2)+\rho^{-1}f_\alpha f_\alpha''+\rho^{-1}f_\beta f_\beta''
\\
&=-\frac{5}{3}\rho^{-3}(f_\alpha f_\alpha'+f_\beta f_\beta')^2+\rho^{-1}(f_\alpha'^2+f_\beta'^2)-\rho^{-\frac{7}{3}}h^2(f_{\alpha}\eta_{\alpha}+f_{\beta}\eta_{\beta})+12\rho^{-1}h^2
\\
&\hphantom{=}~+\rho^{-1}h^{-1}h'(f_{\alpha}f_{\alpha}'+f_{\beta}f_{\beta}')
-2\rho^{-3}f_{\alpha}h(f_{\alpha}f_{\alpha}'+f_{\beta}f_{\beta}').
\end{align*}
We substitute this into the first equation in~\eqref{sys_raw_Poisson_irred}. As a result,
\begin{align}\label{sys_r_Poisson_irred}
\begin{split}
\eta_r
&=
\frac43\rho^{-\frac23}h^{-1}\Big(\rho^{-\frac43} h^2(f_
\alpha\eta_\alpha+f_\beta\eta_\beta)-12h^2 +\frac43\rho^{-2}(f_\alpha f_\alpha' +f_\beta f_\beta')^2-f_\alpha'^2-f_\beta'^2+9h^2\Big).
\end{split}
\end{align}








Let us state the Cauchy conditions~\eqref{initial_conditions} in terms of $h$, $f_\alpha$ and~$f_\beta$. The forms $\psi$, $\psi^+$ and $\psi^-$ satisfy
\begin{align}\label{Cauchyirreddata}
\psi|_{\Sigma}=X\,dr\wedge\omega+A\alpha+B\beta,\qquad \psi^+=A'\alpha+B'\beta,
\qquad \psi^-
=X'\omega,
\end{align}
where $X$, $X'$, $A$, $A'$, $B$ and $B'$ are real numbers. The positivity of $\phi$ requires that $X\ne0$ and $A^2+B^2>0$. Throughout the rest of Section~\ref{subsubsec_exist_irred}, we assume without loss of generality that~$\nu=dr$. This implies
\begin{align} \label{Assumption on the metric - irreducible case}
\rho_0^{\frac23}X^{-1}=1,\qquad \rho_0=\sqrt{A^2+B^2}.
\end{align}
Before we state the Cauchy conditions it will be useful to understand the conditions~\eqref{suff_con} in the current context; to that end we have the following lemma.
\begin{lemma}\label{suff_con - irreducible case}
Conditions~(\ref{suff_con}) on the data $\eta$, $\psi$, $\psi^+$ and $\psi^-$ are equivalent to
\begin{align*}
-4(AA'+BB') &= 3\rho_0^\frac{4}{3}X',
\\
(A\eta_\alpha(0)+B\eta_\beta(0)) -\frac34\rho_0^{\frac{4}{3}}\eta_r(0) &= 12\rho_0^{-\frac{2}{3}}B^2 +(A'^2+B'^2) + 6\rho_0^{\frac{2}{3}}A'\\
&\hphantom{=}~-\frac{4}{3}\rho_0^{-2}(AA'+BB')^2 - 8\rho_0^{-\frac{4}{3}}A(AA'+BB').
\end{align*}
\end{lemma}

\begin{proof}
Clearly,
\begin{align*}
\|\omega\|^2
=3\rho_0^{-\frac{4}{3}}, \qquad \|\omega^2\|^2=12\rho_0^{-\frac{8}{3}}, \qquad
\|\alpha\|^2=\|\beta\|^2 = 4\rho_0^{-2},
\end{align*}
where $\|\cdot\|=\langle\cdot,\cdot\rangle^{\frac12}$. The first identity follows from the first line in~\eqref{suff_con} and the equalities
\begin{align*}
    \langle\psi^+,\psi_{\scparal}\rangle &= \langle A'\alpha + B'\beta, A\alpha+B\beta\rangle= 4\rho_0^{-2}(AA'+BB'),
    \\
    \langle\psi^-,\psi_{\scperp}\rangle &= \langle X'\omega,X\omega\rangle = 3\rho_0^{-\frac{4}{3}}XX'.
\end{align*}
To prove the second, observe that, by symmetry, integration over $\Sigma$ amounts to multiplication by the volume of~$\Sigma$. Therefore, we can omit the integrals in~\eqref{suff_con} and replace the $L_2$ norms by pointwise norms.
With~\eqref{Assumption on the metric - irreducible case} in mind, the left-hand side of the second identity in~\eqref{suff_con} becomes
\begin{align*}
\langle\eta,\psi\rangle - 2\langle\eta_{\scperp},\psi_{\scperp}\rangle = \langle\eta_{\scparal},\psi_{\scparal}\rangle - \langle\eta_{\scperp},\psi_{\scperp}\rangle 
&=
\langle\eta_\alpha(0)\,\alpha+\eta_\beta(0)\,\beta, A\alpha+B\beta\rangle - \langle\eta_r(0)\,\omega,X \omega\rangle 
\\
&= 4\rho_0^{-2}(A\eta_\alpha(0) + B\eta_\beta(0))- 3\rho_0^{-\frac{2}{3}}\eta_r(0).
\end{align*}
Analogously, using~\eqref{exterior derivative for irreducible case} and~\eqref{SigmaHodge - irreducible}, we find 
\begin{align*}
\|d^\Sigma\psi_{\scparal}\|^2+\|\delta^\Sigma\psi_{\scparal}\|^2 &= 4\|B\omega^2\|^2+16\|A\rho_0^{-\frac{2}{3}}\omega\|^2 = 48B^2\rho_0^{-\frac{8}{3}}+48A^2\rho_0^{-\frac{8}{3}},
\\
\|d^\Sigma\psi_{\scperp}\|^2 + \|\delta^\Sigma\psi_{\scperp}\|^2 &= 9\|X\alpha\|^2 = 36\rho_0^{-\frac{2}{3}},
\\
\|\psi^++d^\Sigma\psi_{\scperp}\|^2 &= \|(A'+3X)\alpha + B'\beta\|^2 =4\rho_0^{-2}(A'^2+B'^2)+36\rho_0^{-\frac{2}{3}}+24\rho_0^{-\frac{4}{3}}A',
\\
\|\psi^--\delta^\Sigma\psi_{\scparal}\|^2 &= \|(X'-4\rho_0^{-\frac{2}{3}}A)\omega\|^2=3\rho_0^{-\frac{4}{3}}X'^2 + 48\rho_0^{-\frac{8}{3}}A^2 -24\rho_0^{-2}AX'.
\end{align*}
Together with the first identity in~\eqref{suff_con}, these calculations imply that the right-hand side of the second identity equals
\begin{align*}
48\rho_0^{-\frac{8}{3}}B^2 + 4\rho_0^{-2}(A'^2+B'^2) +24\rho_0^{-\frac{4}{3}}A' - \frac{16}{3}\rho_0^{-4}(AA'+BB')^2 - 32\rho_0^{-\frac{10}{3}}A(AA'+BB').
\end{align*}
\end{proof}
In light of~\eqref{3formisotropyirrep}, \eqref{eq_rho_prime} and Lemma \ref{lem_starddelta_irred}, the Cauchy conditions~\eqref{initial_conditions} become
\begin{align}\label{Cauchyirreducible}
\begin{split}
f_\alpha(0)&=A,\quad f_\beta(0)=B, \quad h(0)=X,
\\
f_\alpha'(0)-3X&=A',\quad 
f_\beta'(0)=B',\\ 
-\frac43\rho_0^{-\frac43}(AA'+BB')&=X'.
\end{split}
\end{align}
The last of these equalities follows from our hypotheses; see Lemma~\ref{suff_con - irreducible case}. Thus, to prove the existence portion of Theorem~B, we need to solve system~\eqref{sys_ab_Poisson_irred}--\eqref{sys_r_Poisson_irred} subject to
\begin{align*}
f_\alpha(0)=A,\quad f_\beta(0)=B,\quad
h(0)=X,\quad f_\alpha'(0)=A'+3X,\quad f_\beta'(0)=B'.
\end{align*}
To do so, we use a variant of DeTurck's trick; cf.~\cite[Theorem~2]{P13}.

Because derivatives of $h$ are absent from~\eqref{sys_r_Poisson_irred}, standard ODE results do not yield short-time solutions to~\eqref{sys_ab_Poisson_irred}--\eqref{sys_r_Poisson_irred}. To overcome this difficulty, consider a $G_2$-structure 
\begin{align}\label{hatphiireep}
\hat\phi=\hat h\,dr\wedge\omega +\hat f_\alpha\,\alpha+\hat f_\beta\,\beta
\end{align}
with $\hat h'(0)=0$ and $\hat h''(0)\ne0$. For each such $G_2$-structure, there are positive numbers $\epsilon_1$, $\epsilon_2$, $\epsilon_1'$ and $\epsilon_2'$ such that the map
\begin{align*}
\Phi:(-\epsilon_1,\epsilon_2)\times G/K\to (-\epsilon_1',\epsilon_2')\times G/K, \qquad \Phi(r,x)=(\hat{h}'(r),x),
\end{align*}
is a diffeomorphism between two neighbourhoods of~$\Sigma$. If $\hat\phi$ solves
\begin{align}\label{Poisson-DeTurck_irred}
\Delta_{\hat\phi}\hat\phi=\Phi^*\eta,
\end{align}
then $\phi=(\Phi^{-1})^*\hat\phi$ solves~\eqref{Poisson_equation}. Moreover, suppose that 
\begin{align}\label{inicon_DeTurck_irred}
\hat\phi|_\Sigma=(\Phi^*\psi)|_\Sigma,\qquad 
i^*(\iota_{\hat{\mathbf n}} d\hat\phi)=\psi^+,\qquad (\hat\delta\hat\phi)_{\scparal}=\psi^-,
\end{align}
 where $i$ is the embedding $\Sigma\to M$, the vector field $\hat{\mathbf n}$ is the image of $\mathbf n$ under $d\Phi^{-1}$, and $\hat\delta$ is the codifferential with respect to the metric induced by~$\hat\phi$. In this case, $\phi=(\Phi^{-1})^*\hat\phi$ satisfies the Cauchy conditions~\eqref{initial_conditions}. Thus, the proof of the existence portion of Theorem~B will be complete if we produce $\hat\phi$ of the form~\eqref{hatphiireep}  solving~\eqref{Poisson-DeTurck_irred}--\eqref{inicon_DeTurck_irred}. Our next step is to state this problem in terms of $\hat h$, $\hat f_\alpha$ and~$\hat f_\beta$.

The pullback $\Phi^*\eta$ is given by
\begin{align*}
\Phi^*\eta=\hat h''\hat\eta_r\,dr\wedge\omega +\hat\eta_\alpha\,\alpha+\hat\eta_\beta\,\beta,
\end{align*}
where $\hat\eta_r=\eta_r\circ \hat h'$, $\hat\eta_\alpha=\eta_\alpha\circ \hat h'$ and $\hat\eta_\beta=\eta_\beta\circ \hat h'$. Substituting $\Phi^*\eta$ for $\eta$ and $\hat{\phi}$ for $\phi$ in~\eqref{sys_ab_Poisson_irred} and~\eqref{sys_r_Poisson_irred}, and noting that  $X\eta_r(0)\ne 0$ by~\eqref{nondeg_assumtion}, 
we conclude that~\eqref{Poisson-DeTurck_irred} is equivalent to
\begin{align}\label{sys_DeTurck_irred}
\begin{split}
\hat f_\alpha''&=-\hat\rho^{-\frac43}\hat h^2\hat\eta_\alpha+12\hat \rho^{-2}\hat f_\alpha \hat h^2-\frac23\hat \rho^{-2}(\hat f_\alpha'+3\hat h)(\hat f_\alpha \hat f_\alpha'+\hat f_\beta\hat f_\beta')+\hat f_\alpha'\hat h^{-1}\hat h',
\\
\hat f_\beta''&=-\hat\rho^{-\frac43}\hat h^2\hat\eta_\beta+12\hat\rho^{-2}\hat f_\beta \hat h^2
-\frac23\hat\rho^{-2}\hat f_\beta'(\hat f_\alpha\hat f_\alpha'+\hat f_\beta\hat f_\beta')
+\hat f_\beta'\hat h^{-1}\hat h',
\\
\hat h''&=\frac43\hat\rho^{-\frac23}\hat h^{-1}\hat\eta_r^{-1}(\hat\rho^{-\frac43}\hat h^2(\hat f_\alpha\hat\eta_\alpha
+\hat f_\beta\hat\eta_\beta)
-12\hat\rho^{-2}\hat h^2(\hat f_\alpha^2+\hat f_\beta^2))
\\
&\hphantom{=}~+\frac43\hat\rho^{-\frac23}\hat h^{-1}\hat\eta_r^{-1}\Big(\frac43\hat\rho^{-2}(\hat f_\alpha\hat f_\alpha'+\hat f_\beta\hat f_\beta')^2-\hat f_\alpha'^2-\hat f_\beta'^2+9\hat h^2\Big),
\end{split}
\end{align}
with $\hat\rho=\sqrt{\hat f_\alpha^2+\hat f_\beta^2}$. Using~\eqref{hatphiireep}, Lemma \ref{lem_starddelta_irred}, the formula $\hat{\mathbf{n}}=(h''(0))^{-1}\mathbf{n}$ and Lemma~\ref{suff_con - irreducible case}, we find that conditions~\eqref{inicon_DeTurck_irred} hold if and only if
\begin{align*}
\hat f_\alpha(0)=A,\quad \hat f_\beta(0)=B,\quad \hat h(0)=\hat h''(0)X,\quad \hat f_\alpha'(0)-3\hat h(0)=\hat h''(0)A',\quad \hat{f}_\beta'(0)=\hat h''(0)B'.
\end{align*}
Let us choose a positive number $\kappa$ and supplement these equalities with the requirement $\hat h(0)=\kappa$. As a consequence,
\begin{align*}
\hat h''(0)=\kappa X^{-1}\ne0.
\end{align*}
We also demand that $\hat h'(0)=0$ to ensure that the map $r\mapsto \hat h'(r)$ defines a diffeomorphism between neighbourhoods of~$\Sigma$. Thus, we aim to solve system~\eqref{sys_DeTurck_irred} under the conditions
\begin{align}\label{inicon_DeTurck_ODE_ir}
\begin{split}
\hat f_\alpha(0)=A,\quad \hat f_\beta(0)&=B,\quad \hat h(0)=\kappa ,\\
\hat f_\alpha'(0)&=\kappa (X^{-1}A'+3),\quad \hat f_\beta'(0)=\kappa X^{-1}B',\quad \hat h'(0)=0,\quad \hat h''(0)=\kappa X^{-1}.
\end{split}
\end{align}
This would yield the existence of $\hat\phi$ satisfying~\eqref{Poisson-DeTurck_irred}--\eqref{inicon_DeTurck_irred}.

Let us employ classical ODE results. The Picard--Lindel\"of theorem implies that~\eqref{Poisson-DeTurck_irred} has a solution for which the first six equalities of~\eqref{inicon_DeTurck_ODE_ir} hold. We claim that, with our assumptions, this solution necessarily satisfies the seventh equality. Indeed, the formula for~$\hat{h}''$ in~\eqref{sys_DeTurck_irred} implies
\begin{align*}
\hat h''(0)&=\frac43\kappa \rho_0^{-\frac23}\eta_r^{-1}(0)(\rho_0^{-\frac43}(A\eta_\alpha(0)+B\eta_\beta(0))
-12)
\\
&\hphantom{=}~+\frac43\kappa \rho_0^{-\frac23}\eta_r^{-1}(0)\Big(
\frac43\rho_0^{-2}(A(X^{-1}A'+3)+X^{-1}BB')^2-(X^{-1}A'+3)^2-X^{-2}B'^2+9\Big).
\end{align*}
Consequently, $\hat{h}''(0)$ equals $\kappa X^{-1}$ if and only if
\begin{align*}
\begin{split}
\rho_0^{-\frac43}X^2(A\eta_\alpha(0)&+B\eta_\beta(0))-\frac34\rho_0^{\frac23}X\eta_r(0)-12\rho_0^{-2}X^2B^2
\\
&+8\rho_0^{-2}XA(AA'+BB')-6XA'
+\frac43\rho_0^{-2}\left(AA'+BB'\right)^2-A'^2-B'^2=0.
\end{split}
\end{align*}
However, this identity follows from~\eqref{Assumption on the metric - irreducible case} and Lemma~\ref{suff_con - irreducible case}. Thus, the proof of the existence part of Theorem~B is complete.

\subsubsection{Uniqueness}\label{subsec_uniq_irred}

Suppose that $\phi_1$ and $\phi_2$ are solutions to~\eqref{Poisson_equation}--\eqref{initial_conditions} in a neighbourhood of~$\Sigma$ with
\begin{equation*}
\phi_i=h_i\,dr\wedge\omega + f_{\alpha i}\,\alpha + f_{\beta i}\,\beta,\qquad i=1,2.
\end{equation*}
Choose $\kappa>0$ and consider the initial-value problems
\begin{align}\label{eq_harm_flow}
\hat h_i''=\frac{\hat h_i}{h_i\circ\hat h_i'},\qquad \hat h_i(0)=\kappa ,\qquad \hat h_i'(0)=0,\qquad i=1,2.
\end{align}
Since $\phi_1$ and $\phi_2$ are $G_2$-structures, the values of the functions $h_1$ and $h_2$ are nonzero. By the Picard--Lindel\"of theorem, there exists an interval $(-\epsilon,\epsilon)$ on which each of the problems~\eqref{eq_harm_flow} has a unique solution. We can choose $\epsilon$ small enough to ensure that $\hat h_i$ and $\hat h_i''$ are nonzero on this interval.

Let $\Phi_i$ be the diffeomorphisms between neighbourhoods of $\Sigma$ induced by the maps $r\mapsto\hat h_i'(r)$. The $G_2$-structures $\hat\phi_i=\Phi_i^*\phi_i$ satisfy~\eqref{Poisson-DeTurck_irred}--\eqref{inicon_DeTurck_irred}. This fact, together with~\eqref{eq_harm_flow}, implies that the components $\hat h_i$, $\hat f_{\alpha i}$ and $\hat f_{\beta i}$ of each $\hat\phi_i$ solve system~\eqref{sys_DeTurck_irred} under conditions~\eqref{inicon_DeTurck_ODE_ir}. Invoking the Picard--Lindel\"of theorem again, we conclude that
\begin{align*}
\hat h_1=\hat h_2,\qquad \hat f_{\alpha1}=\hat f_{\alpha2},\qquad \hat f_{\beta1}=\hat f_{\beta2}.
\end{align*}
Therefore, $\Phi_1=\Phi_2$ and $\hat\phi_1=\hat\phi_2$, whence
\begin{align*}
\phi_1=(\Phi_1^{-1})^*\hat\phi_1=(\Phi_2^{-1})^*\hat\phi_2=\phi_2
\end{align*}
in a neighbourhood of $\Sigma$. This 
proves the uniqueness portion of Theorem~B.

\subsection{Orbit type $SU(3)/\mathbb T^2$}\label{subsec_flag_mf}
Assume that $G$ is the group $SU(3)$ and $K$ is the subgroup $\mathbb T^2$ of diagonal matrices.
In this case, the isotropy representation of $G/K$ 
splits into three inequivalent irreducible 2-dimensional summands.

\subsubsection{Invariant forms on the orbit}\label{subsubsec_inv_orb_red}

Suppose that $\mathfrak m$ is the subspace of $\mathfrak{su}(3)$ 
with generic element of the form
\begin{align*}
X= \begin{pmatrix}
0&u_1&u_2\\
-\overline{u}_1&0&u_3\\
-\overline{u}_2&-\overline{u}_3&0
\end{pmatrix},\qquad u_1,u_2,u_3\in\mathbb C.
\end{align*}
For every $t=\text{diag}(e^{i\theta_1},e^{i\theta_2},e^{i\theta_3})\in K$, the sum $\theta_1+\theta_2+\theta_3$ equals 0 modulo $2\pi$, and
\begin{align}\label{eq_Ad_expl}
\Ad^K(t)(X)=t X t^{-1}=\begin{pmatrix}
0 & e^{i(\theta_1 - \theta_2)} u_1 & e^{i(\theta_1 - \theta_3)} u_2 \\
-e^{i(\theta_2 - \theta_1)} \overline{u}_1 & 0 & e^{i(\theta_2 - \theta_3)} u_3 \\
-e^{i(\theta_3 - \theta_1)} \overline{u}_2 & -e^{i(\theta_3 - \theta_2)} \overline{u}_3 & 0
\end{pmatrix}.
\end{align}
Let $\{e_j\}_{j=1}^6$ be the basis of $\mathfrak{m}$ given by
\begin{align*}
e_1=\frac{1}{2i}\lambda_7, \qquad e_2=\frac{1}{2i}\lambda_6, \qquad e_3=-\frac{1}{2i}\lambda_5, \qquad e_4=\frac{1}{2i}\lambda_4,\qquad e_5=\frac1{2i}\lambda_2, \qquad e_6=\frac{1}{2i}\lambda_1,
\end{align*}
where $\lambda_j$ are the Gell-Mann matrices; cf.~\cite{CS02}. Denote the corresponding dual basis~$\{e^j\}_{j=1}^6$.
The space $\mathfrak m$ admits the $\tilde K$-invariant decomposition
\begin{align*}
\mathfrak m=\mathfrak m_1\oplus\mathfrak m_2\oplus\mathfrak m_3,
\quad
\mathfrak{m}_1=\text{span}\{e_1,e_2\}, \quad \mathfrak{m}_2=\text{span}\{e_3,e_4\}, \quad \mathfrak{m}_3=\text{span}\{e_5,e_6\}.
\end{align*}
It is easy to check that the restrictions of the representation $\Ad^K$ to $\mathfrak m_1$, $\mathfrak m_2$ and $\mathfrak m_3$ are irreducible and pairwise inequivalent. The following result will help us describe the spaces of $G$-invariant forms on~$G/K$.


\begin{lemma}\label{lem_effect}
The dimension of the largest trivial submodule of $[\Lambda_0^{1,1}]$ equals~2.
\end{lemma}

\begin{proof}
The map $\Ad^K:K\to GL(\mathfrak{m})$ is a triple covering with kernel $\{\Id,e^{\frac{2}{3}i\pi}\Id,e^{\frac{4}{3}i\pi}\Id\}$. The image $\tilde K$ of this map is isomorphic to $K$ as a Lie group. Exploiting~\eqref{eq_Ad_expl} and arguing as above, we find that the Lie algebra $\mathfrak{su}(3)$, viewed here as a $\tilde K$-module via the adjoint representation of~$SU(3)$, admits the decomposition
\begin{align*}
\mathfrak{su}(3)\cong\tilde{\mathfrak t}\oplus\tilde{\mathfrak m}_1\oplus\tilde{\mathfrak m}_2\oplus\tilde{\mathfrak m}_3,
\end{align*}
where $\tilde{\mathfrak t}$ is the Lie algebra of $\tilde K$
and each $\tilde{\mathfrak m}_i$ is an irreducible 2-dimensional submodule. The representation of $\tilde K$ on $\tilde{\mathfrak t}$ induced by the adjoint representation of~$SU(3)$ is trivial. Since $\tilde{\mathfrak t}$ is 2-dimensional and $[\Lambda_0^{1,1}]$ is isomorphic to~$\mathfrak{su}(3)$ (see the discussion at the beginning of Section~\ref{sec_Poisson}), this observation completes the proof.
%
%
\end{proof}

An argument based on Schur's lemma shows that every $G$-invariant Riemannian metric on $G/K$ is a linear combination of
\begin{align*}
g_1 &= e^1 \otimes e^1 + e^2 \otimes e^2, \qquad g_2 = e^3 \otimes e^3 + e^4 \otimes e^4, \qquad g_3 =e^5 \otimes e^5 + e^6 \otimes e^6.
\end{align*}
For every $r\in(-1,1)$, the space $\Omega^k(\{r\}\times G/K)^G$ is isomorphic to~$\Lambda^k(\mathfrak{m}^*)^K$. The first line of~\eqref{decomposition of exterior powers under su(3)} and the fact that the representation $\Ad^K$ on $\mathfrak m$ has no trivial subrepresentations imply
$$\Omega^1(\{r\}\times G/K)^G=\{0\}.$$
From the second line,
\begin{align*}
\begin{split}
\Lambda^2(\mathfrak{m}^*) &\cong \mathfrak m \oplus[\Lambda_0^{1,1}]
\oplus\mathbb R.
\end{split}
\end{align*}
Lemma~\ref{lem_effect} shows that
$$
\dim\big(\Lambda^2(\mathfrak{m}^*)^K\big)=3.
$$
One checks easily the $\tilde K$-invariance of the 2-forms
\begin{equation}\label{invariant tensors in coordinates - reducible - omega}
\omega_1 = e^{12},\qquad \omega_2=e^{34}, \qquad \omega_3=e^{56},
\end{equation}
which implies that
$$\Omega^2(\{r\}\times G/K)^G=\text{span}\{\omega_1,\omega_2,\omega_3\}.$$
From the third line of~\eqref{decomposition of exterior powers under su(3)},
\begin{align*}
\Lambda^3(\mathfrak{m}^*) &\cong \mathbb{R} \oplus \mathbb{R} \oplus [[\Lambda^{2,1}_0]]\oplus \mathfrak m.
\end{align*}
Invoking~\eqref{decomp_tens} as in Section~\ref{subsubsection - invariant forms and G2 Laplacian}, we can show that $[[\Lambda^{2,1}_0]]$ has no trivial submodules. This means that
$$
\dim\big(\Lambda^3(\mathfrak{m}^*)^K\big)=2.
$$
Straightforward verification proves the $\tilde K$-invariance of the 3-forms
\begin{equation}\label{invariant tensors in coordinates - reducible - alphabeta}
\alpha = e^{246}-e^{235}-e^{145}-e^{136}, \qquad  \beta = e^{135}-e^{146}-e^{236}-e^{245}.
\end{equation}
Thus,
$$\Omega^3(\{r\}\times G/K)^G=\text{span}\{\alpha,\beta\}.$$

Abbreviating $\omega_j\wedge\omega_k$ to~$\omega_{jk}$, we find
\begin{align}\label{exterior derivative for reducible case}
    d\omega_j=\frac{1}{2}\alpha, \qquad d\beta=-2(\omega_{12}+\omega_{23}+\omega_{13}), \qquad j=1,2,3;
\end{align}
see, e.g.,~\cite[Section~2]{J22} and~\cite{CS02}. The wedge products $\omega_j\wedge \alpha$ and $\omega_j\wedge \beta$ {and the differentials $d\omega_{jk}$} vanish because 
$$
\Omega^5(\{r\}\times G/K)^G\cong\Omega^1(\{r\}\times G/K)^G=\{0\}.
$$
If $*_\mathrm s$ is the Hodge star of the metric $g_1+g_2+g_3$ and $\vol_\mathrm s=e^{123456}$, then
\begin{align*}
    \omega_j\wedge*_\mathrm s\,\omega_j = \omega_k\wedge\omega_l\wedge \omega_m = \vol_\mathrm s, \qquad \alpha\wedge*_\mathrm s\alpha=\alpha\wedge\beta=4\vol_\mathrm s, \qquad j=1,2,3,
\end{align*}
for every permutaion $(k,l,m)$ of the triple $(1,2,3)$.

\subsubsection{Invariant forms on $M$ and the $G_2$-Laplacian}


Every $G$-invariant $3$-form on $M$ appears as
\begin{equation}\label{3formisotropyred}
\phi = h_1\,dr\wedge\omega_1 + h_2\,dr\wedge\omega_2+ h_3\,dr\wedge\omega_3 + f_\alpha\,\alpha + f_\beta\,\beta,
\end{equation}
where $h_1$, $h_2$, $h_3$, $f_\alpha$ and $f_\beta$ are smooth functions on $(-1,1)$. Arguing as in the proof of Lemma~\ref{lem_ind_met_red} below, one can show that $\phi$ is a $G_2$-structure if and only if
\begin{align*}
h_1h_2h_3\ne0,\qquad\rho=\sqrt{f_\alpha^2+f_\beta^2}\ne0,
\end{align*}
and $h_1$, $h_2$ and $h_3$ have the same sign on $(-1,1)$.

\begin{lemma}\label{lem_ind_met_red}
Let $\phi$ be a $G_2$-structure on $M$ given by~(\ref{3formisotropyred}). The metric $g$ induced by $\phi$ satisfies
\begin{equation*}
g=\rho^{-\frac{4}{3}}q^2\,dr\otimes dr + \rho^{\frac{2}{3}}q^{-1}(h_1g_1+h_2g_2+h_3g_3)
\end{equation*}
with $q = (h_1h_2h_3)^{\frac{1}{3}}$.
\end{lemma}

\begin{proof}
One can establish the assertion of the lemma by using the standard formula for the metric associated with a $G_2$-structure; see, e.g.,~\cite[page~539]{B87}. For an alternative proof, define
\begin{align*}
    v^0 = \rho^{-\frac{2}{3}}q\,dr, \quad v^1&=\rho^{-\frac23}h_1^{\frac12}q^{-\frac12}(f_\beta e^1+f_\alpha e^2), \quad  v^2=\rho^{-\frac23}h_1^{\frac12}q^{-\frac12}(f_\beta e^2-f_\alpha e^1), 
    \\v^3&=\rho^\frac{1}{3}h_2^{\frac12}q^{-\frac12}e^3, \quad
 v^4=\rho^\frac{1}{3}h_2^{\frac12}q^{-\frac12}e^4, \quad v^5=\rho^\frac{1}{3}h_3^{\frac12}q^{-\frac12}e^5, \quad v^6=\rho^\frac{1}{3}h_3^{\frac12}q^{-\frac12}e^6.
\end{align*}
Clearly, $\{v^i\}_{i=0}^6$ is a $g$-orthonormal basis of $\mathfrak m^*$. Direct calculation based on~\eqref{3formisotropyred},~\eqref{invariant tensors in coordinates - reducible - omega} and~\eqref{invariant tensors in coordinates - reducible - alphabeta} shows that $\phi$ coincides with the canonical positive 7-form in this basis, i.e.,
\begin{align*}
\phi=v^{012}+v^{034}+v^{056}+v^{135}-v^{146}-v^{236}-v^{245},
\end{align*}
where $v^{i_1\cdots i_l}=v^{i_1}\wedge \cdots \wedge v^{i_l}$. Thus, $g$ must be the metric induced by~$\phi$.
\end{proof}

We compute
\begin{equation}\label{Hodge stars for induced metric - reducible}
    \begin{aligned}
        &*\omega_1 = q^3h_1^{-2}\,dr\wedge\omega_{23}, & &*\omega_2= q^3h_2^{-2}\,dr\wedge\omega_{13},& &*\omega_3 = q^3h_3^{-2}\,dr\wedge\omega_{12},
        \\
        &*(dr\wedge\omega_1) =\rho^{\frac{4}{3}}qh_1^{-2} \omega_{23}, & &*(dr\wedge\omega_2)=\rho^{\frac{4}{3}}qh_2^{-2} \omega_{13},& &*(dr\wedge\omega_3)=\rho^{\frac{4}{3}}qh_3^{-2} \omega_{12}, \\
        &*\alpha = -\rho^{-\frac{2}{3}}q\,dr\wedge\beta, & &*\beta = \rho^{-\frac{2}{3}}q\,dr\wedge\alpha,
    \end{aligned}
\end{equation}
where $*$ is the Hodge star of~$g$, and since $\|dr\|=\rho^{\frac{2}{3}}q^{-1}$,
\begin{equation}\label{SigmaHodge - reducible}
    \begin{aligned}
        &*_0\omega_1 = \rho^{\frac{2}{3}}q^2h_1^{-2}\omega_{23}, & &*_0\omega_2= \rho^{\frac{2}{3}}q^2h_2^{-2}\omega_{13},& &*_0\omega_3 =\rho^{\frac{2}{3}}q^2h_3^{-2}\,dr\wedge\omega_{12}, \\
        &*_0\alpha = \beta, & &*_0\beta = -\alpha, \qquad *_01=\rho^2 \vol_1,
    \end{aligned}
\end{equation}
where $*_0$ is the Hodge star given by $g_0^\Sigma=i^*g$. These equalities help us obtain the following result.

\begin{lemma}\label{lem_calcs_red}
If $\phi$ is a $G_2$-structure on $M$ satisfying~(\ref{3formisotropyred}), then
    \begin{align*}
        *\phi &= \rho^\frac{4}{3}q(h_1^{-1}\omega_{23}+h_2^{-1}\omega_{13}+h_3^{-1}\omega_{12}) - \rho^{-\frac{2}{3}}qf_\alpha\,dr\wedge\beta + \rho^{-\frac{2}{3}}qf_\beta\,dr\wedge\alpha,
        \\
        d\phi &=-\frac{1}{2}(h_1+h_2+h_3-2f_\alpha')\,dr\wedge\alpha + f_\beta'\,dr\wedge\beta -2f_\beta(\omega_{12}+\omega_{23}+\omega_{13}),
        \\
        \delta\phi &= -q^{-3}h_1^2((\rho^\frac{4}{3}qh_1^{-1})'-2\rho^{-\frac{2}{3}}qf_\alpha)\,\omega_1
        \\
        &\hphantom{=}~-q^{-3}h_2^2((\rho^\frac{4}{3}qh_2^{-1})'-2\rho^{-\frac{2}{3}}qf_\alpha)\,\omega_2
        -q^{-3}h_3^2((\rho^\frac{4}{3}qh_3^{-1})'-2\rho^{-\frac{2}{3}}qf_\alpha)\,\omega_3.
    \end{align*}
\end{lemma}

\begin{proof}
The first equality is a direct consequence of~\eqref{Hodge stars for induced metric - reducible}. The second follows from a simple calculation based on~\eqref{exterior derivative for reducible case}. Next, we exploit both~\eqref{Hodge stars for induced metric - reducible} and~\eqref{exterior derivative for reducible case} to find
\begin{align*}
\delta\phi 
&= 
-* d*(h_1\,dr\wedge\omega_1 + h_2\,dr\wedge\omega_2+h_3\,dr\wedge\omega_3+f_\alpha\alpha+f_\beta\beta)
\\
&= -* d(\rho^\frac{4}{3}q(h_1^{-1}\omega_{23}+h_2^{-1}\omega_{13}+h_3^{-1}\omega_{12}) - \rho^{-\frac{2}{3}}qf_\alpha\,dr\wedge\beta + \rho^{-\frac{2}{3}}qf_\beta\,dr\wedge\alpha)
\\
&=-*(((\rho^\frac{4}{3}qh_1^{-1})'-2\rho^{-\frac{2}{3}}qf_\alpha)\,dr\wedge\omega_{23}) 
\\
&\hphantom{=}~-*\,(((\rho^\frac{4}{3}qh_2^{-1})'-2\rho^{-\frac{2}{3}}qf_\alpha)\,dr\wedge\omega_{13})
-*\,(((\rho^\frac{4}{3}qh_3^{-1})'-2\rho^{-\frac{2}{3}}qf_\alpha)\,dr\wedge\omega_{12})
\\
&=-q^{-3}h_1^2((\rho^\frac{4}{3}qh_1^{-1})'-2\rho^{-\frac{2}{3}}qf_\alpha)\,\omega_1
\\
&\hphantom{=}~-q^{-3}h_2^2((\rho^\frac{4}{3}qh_2^{-1})'-2\rho^{-\frac{2}{3}}qf_\alpha)\,\omega_2
-q^{-3}h_3^2((\rho^\frac{4}{3}qh_3^{-1})'-2\rho^{-\frac{2}{3}}qf_\alpha)\,\omega_3.
\end{align*}
\end{proof}

Let us compute the $G_2$-Laplacian.

\begin{proposition}\label{G2 Laplacian - reducible case}
If $\phi$ is a $G_2$-structure on $M$ satisfying~(\ref{3formisotropyred}), then
    \begin{align*}
        \Delta_\phi\phi &= (\rho^{-\frac{2}{3}}q^{-2}h_1^2(h_1+h_2+h_3 - 2f_\alpha')-(q^{-3}h_1^2((\rho^\frac{4}{3}qh_1^{-1})'-2\rho^{-\frac{2}{3}}q f_\alpha))' )\,dr\wedge\omega_1
        \\
        &\hphantom{=}~+ (\rho^{-\frac{2}{3}}q^{-2}h_2^2(h_1+h_2+h_3 - 2f_\alpha')-(q^{-3}h_2^2((\rho^\frac{4}{3}qh_2^{-1})'-2\rho^{-\frac{2}{3}}q f_\alpha))')\,dr\wedge\omega_2
        \\
        &\hphantom{=}~+(\rho^{-\frac{2}{3}}q^{-2}h_3^2(h_1+h_2+h_3 - 2f_\alpha')-(q^{-3}h_3^2((\rho^\frac{4}{3}qh_3^{-1})'-2\rho^{-\frac{2}{3}}q f_\alpha))')\,dr\wedge\omega_3
        \\
        &\hphantom{=}~-\frac{1}{2}q^{-3}(h_1^2(\rho^\frac{4}{3}qh_1^{-1})' + h_2^2(\rho^\frac{4}{3}qh_2^{-1})' + h_3^2(\rho^\frac{4}{3}qh_3^{-1})')\alpha 
        \\
        &\hphantom{=}~+\rho^{\frac{2}{3}}q^{-1}\Big(\rho^{-\frac{4}{3}}q^{-1}f_\alpha(h_1^2 + h_2^2 + h_3^2)
        +\frac{1}{2}(\rho^\frac{2}{3}q^{-1}(h_1+h_2+h_3-2f_\alpha'))'\Big)\alpha
        \\
        &\hphantom{=}~- \rho^\frac{2}{3}q^{-1}((\rho^{\frac{2}{3}}q^{-1}f_\beta')' -\rho^{-\frac{4}{3}}q^{-1}f_\beta(h_1^2+h_2^2+h_3^2))\beta.
    \end{align*}
\end{proposition}

\begin{proof}
We exploit Lemma~\ref{lem_calcs_red} and the first equality in~\eqref{exterior derivative for reducible case} to find
\begin{align*}
d\delta\phi &=-(q^{-3}h_1^2((\rho^\frac{4}{3}qh_1^{-1})'-2\rho^{-\frac{2}{3}}qf_\alpha))'\,dr\wedge\omega_1 
\\
&\hphantom{=}~-(q^{-3}h_2^2((\rho^\frac{4}{3}qh_2^{-1})'-2\rho^{-\frac{2}{3}}qf_\alpha))'\,dr\wedge\omega_2
- (q^{-3}h_3^2((\rho^\frac{4}{3}qh_3^{-1})'-2\rho^{-\frac{2}{3}}qf_\alpha))'\,dr\wedge\omega_3
\\
&\hphantom{=}~-\frac{1}{2}q^{-3}(h_1^2(\rho^\frac{4}{3}qh_1^{-1})' + h_2^2(\rho^\frac{4}{3}qh_2^{-1})' + h_3^2(\rho^\frac{4}{3}qh_3^{-1})' - 2\rho^{-\frac{2}{3}}qf_\alpha(h_1^2 + h_2^2 + h_3^2))\,\alpha.
\end{align*}
Similarly, with~\eqref{Hodge stars for induced metric - reducible} in mind,
\begin{align*}
        \delta d\phi 
        &= *\,d*
        \Big(-\frac{1}{2}(h_1+h_2+h_3-2f_\alpha')\,dr\wedge\alpha + f_\beta'\,dr\wedge\beta -2f_\beta(\omega_{12}+\omega_{23}+\omega_{13})\Big)
        \\
        &=*\,d\Big(-\frac12\rho^\frac{2}{3}q^{-1}(h_1+h_2+h_3-2f_\alpha')\beta - \rho^\frac{2}{3}q^{-1}f_\beta'\,\alpha - 2\rho^{-\frac{4}{3}}q^{-1}f_\beta\,dr\wedge(h_1^2\omega_1+h_2^2\omega_2+h_3^2\omega_3)\Big)
        \\
        &=*\,(\rho^\frac{2}{3}q^{-1}(h_1+h_2+h_3-2f_\alpha')(\omega_{12}+\omega_{23}+\omega_{13}))
        \\
        &\hphantom{=}
        -*\,\Big(((\rho^\frac{2}{3}q^{-1}f_\beta')' - \rho^{-\frac{4}{3}}q^{-1}f_\beta(h_1^2+h_2^2+h_3^2))\,dr\wedge\alpha
+\frac12(\rho^\frac{2}{3}q^{-1}(h_1+h_2+h_3-2f_\alpha'))'\,dr\wedge\beta\Big)
        \\
        &=\rho^{-\frac{2}{3}}q^{-2}(h_1+h_2+h_3-2f_\alpha')(h_1^2\:dr\wedge\omega_1+h_2^2\:dr\wedge\omega_2+h_3^2\:dr\wedge\omega_3)
        \\
        &\hphantom{=}~+\frac{1}{2}\rho^\frac{2}{3}q^{-1}(\rho^\frac{2}{3}q^{-1}(h_1+h_2+h_3-2f_\alpha'))'\alpha
- (\rho^\frac{2}{3}q^{-1}(\rho^{\frac{2}{3}}q^{-1}f_\beta')' - \rho^{-\frac{2}{3}}q^{-2}f_\beta(h_1^2+h_2^2+h_3^2))\,\beta.
\end{align*}
The assertion of the lemma follows by combining these calculations.
\end{proof}

\subsubsection{Existence}\label{subsubsec_exist_red}

The 3-form on the right-hand side of the Poisson equation~\eqref{Poisson_equation} appears as
\begin{align*}
\eta=\eta_1\,dr\wedge\omega_1+\eta_2\,dr\wedge\omega_2+\eta_3\,dr\wedge\omega_3 +\eta_\alpha\,\alpha+\eta_\beta\,\beta,
\end{align*}
where $\eta_1$, $\eta_2$, $\eta_3$, $\eta_\alpha$ and $\eta_\beta$ are functions on~$(-1,1)$.
Proposition~\ref{G2 Laplacian - reducible case} reduces~\eqref{Poisson_equation} to
\begin{align}\label{sys_raw_Poisson_red}
\begin{split}
\rho^{-\frac{2}{3}}q^{-2}h_1^2(h_1+h_2+h_3 - 2f_\alpha')-(q^{-3}h_1^2((\rho^\frac{4}{3}qh_1^{-1})'-2\rho^{-\frac{2}{3}}q f_\alpha))'&=\eta_1,
\\
\rho^{-\frac{2}{3}}q^{-2}h_2^2(h_1+h_2+h_3 - 2f_\alpha')-(q^{-3}h_2^2((\rho^\frac{4}{3}qh_2^{-1})'-2\rho^{-\frac{2}{3}}q f_\alpha))'&=\eta_2,
\\
\rho^{-\frac{2}{3}}q^{-2}h_3^2(h_1+h_2+h_3 - 2f_\alpha')-(q^{-3}h_3^2((\rho^\frac{4}{3}qh_3^{-1})'-2\rho^{-\frac{2}{3}}q f_\alpha))'&=\eta_3,
\\
-\frac{1}{2}q^{-3}(h_1^2(\rho^\frac{4}{3}qh_1^{-1})' + h_2^2(\rho^\frac{4}{3}qh_2^{-1})' + h_3^2(\rho^\frac{4}{3}qh_3^{-1})')+\rho^{-\frac{2}{3}}q^{-2}f_\alpha(h_1^2 + h_2^2 + h_3^2)
\\
\hphantom{=}~+\frac{1}{2}\rho^{\frac{2}{3}}q^{-1}(\rho^\frac{2}{3}q^{-1}(h_1+h_2+h_3-2f_\alpha'))'&=\eta_\alpha,
\\
-\rho^\frac{2}{3}q^{-1}((\rho^{\frac{2}{3}}q^{-1}f_\beta')' - \rho^{-\frac{4}{3}}q^{-1}f_\beta(h_1^2+h_2^2+h_3^2))&=\eta_\beta.
\end{split}
\end{align}
Rearranging the terms in these equations yields
\begin{align*}
-q''q^2+ h_1''h_2h_3 &= \rho^{-\frac{4}{3}}q^2h_2h_3\eta_1 - h_1G_{(i)} - G_{(ii)} - h_1'h_2h_3G_{(iii)} - h_1'G_{(iv)},
\\
    -q''q^2+ h_1h_2''h_3 &= \rho^{-\frac{4}{3}}q^2h_1h_3\eta_2 - h_2G_{(i)} - G_{(ii)} - h_1h_2'h_3G_{(iii)} - h_2'G_{(iv)},
    \\
    -q''q^2+ h_1h_2h_3'' &= \rho^{-\frac{4}{3}}q^2h_1h_2\eta_3 - h_3G_{(i)} - G_{(ii)} - h_1h_2h_3'G_{(iii)} - h_3'G_{(iv)}
\\
f_\alpha'' &= - \rho^{-\frac{4}{3}}q^2\eta_\alpha+G_\alpha,
\\
f_\beta'' &= - \rho^{-\frac{4}{3}}q^2\eta_\beta+G_\beta.
\end{align*}
Here, $G_{(i)}$, $G_{(ii)}$, $G_{(iii)}$, $G_{(iv)}$, $G_\alpha$ and $G_\beta$ are algebraic functions of $h_1$, $h_2$, $h_3$, $h_1'$, $h_2'$, $h_3'$, $f_\alpha$, $f_\beta$, $f_\alpha'$ and $f_\beta'$ which are symmetric in $h_1$, $h_2$ and $h_3$ as well as in $h_1'$, $h_2'$ and $h_3'$. We add the first three lines and observe that
\begin{align*}
q''=((h_1h_2h_3)^{\frac{1}{3}})''=\Big(\frac{1}{3}(h_1h_2h_3)^{-\frac{2}{3}}(h_1h_2h_3)'\Big)' = -\frac{2}{9}q^{-5}(h_1h_2h_3)'+\frac{1}{3}q^{-2}(h_1h_2h_3)''
\end{align*}
to obtain
\begin{align*}
    -2(h_1'h_2h_3 + h_1h_2'h_3 + h_1h_2h_3') = \rho^{-\frac{4}{3}}q^2(h_2h_3\eta_1 + h_1h_3\eta_2 + h_1h_2\eta_3) + G_*.
\end{align*}
Again, $G_*$ is an algebraic function of $h_1$, $h_2$, $h_3$, $h_1'$, $h_2'$, $h_3'$, $f_\alpha$, $f_\beta$, $f_\alpha'$ and $f_\beta'$ which is symmetric in $h_1$, $h_2$ and $h_3$ as well as in their derivatives. We conclude that system~\eqref{sys_raw_Poisson_red} is equivalent to the system
\begin{align}\label{Poisson - reducible case}
\begin{split}
-q''q^2+ h_1''h_2h_3 &= \rho^{-\frac{4}{3}}q^2h_2h_3\eta_1 - h_1G_{(i)} - G_{(ii)} - h_1'h_2h_3G_{(iii)} - h_1'G_{(iv)},
\\
-q''q^2+ h_1h_2''h_3 &= \rho^{-\frac{4}{3}}q^2h_2h_3\eta_2 - h_2G_{(i)} - G_{(ii)} - h_1h_2'h_3G_{(iii)} - h_2'G_{(iv)},
\\
-2(h_1'h_2'h_3+h_1'h_2h_3'+h_1h_2'h_3') &= \rho^{-\frac{4}{3}}q^2(h_2h_3\eta_1 +h_1h_3\eta_2 + h_1h_2\eta_3) + G_*,
\\
 f_\alpha'' &= - \rho^{-\frac{4}{3}}q^2\eta_\alpha + G_\alpha,
\\
f_\beta'' &= - \rho^{-\frac{4}{3}}q^2\eta_\beta+G_\beta.
\end{split}
\end{align}
The vector of the leading terms appears as
\begin{align}\label{sys_matr_red}
\frac13M(h_1,h_2,h_3)\left(
\begin{matrix}
h_1'' \\
h_2'' \\
h_3'' \\
f_\alpha'' \\
f_\beta''
\end{matrix}\right),\qquad
M(x,y,z)=\left(
\begin{matrix}
2yz & -xz & -xy & \hphantom{x}0\hphantom{y} & \hphantom{x}0\hphantom{y} \\
-yz & 2xz & -xy & \hphantom{x}0\hphantom{y} & \hphantom{x}0\hphantom{y} \\
0 & 0 & 0 & \hphantom{x}0\hphantom{y} & \hphantom{x}0\hphantom{y} \\
0 & 0 & 0 & \hphantom{x}1\hphantom{y} & \hphantom{x}0\hphantom{y} \\
0 & 0 & 0 & \hphantom{x}0\hphantom{y} & \hphantom{x}1\hphantom{y} \\
\end{matrix}\right).
\end{align}

Let us state the Cauchy conditions~\eqref{initial_conditions} in terms of $h_1$, $h_2$, $h_3$, $f_\alpha$ and~$f_\beta$. The forms $\psi$, $\psi^+$ and $\psi^-$ satisfy
\begin{align*}
\psi=X\,dr\wedge\omega_1+Y\,dr\wedge\omega_2+Z\,dr\wedge\omega_3&+A\alpha+B\beta,
\\ 
\psi^+&=A'\alpha+B'\beta,
\qquad \psi^-
=X'\omega_1+Y'\omega_2+Z'\omega_3,
\end{align*}
where $X$, $X'$, $Y$, $Y'$, $Z$, $Z'$, $A$, $A'$, $B$ and $B'$ are real numbers. The positivity of $\psi$ requires that $XYZ\ne0$ and $A^2+B^2>0$. 
Since~$\nu=dr$ on $\Sigma$,
\begin{align}\label{Assumption on the metric - reducible case}
\rho_0^2 = q^3(0) = XYZ,\qquad \rho_0^2=\sqrt{A^2+B^2}.
\end{align}

\begin{lemma}\label{suff_con - reducible case}
    Conditions~(\ref{suff_con}) on the data $\eta$, $\psi$, $\psi^+$ and $\psi^-$ are equivalent to
    \begin{align*}
        -4(AA'+BB') &= \rho^2_0(X'X^{-1} + Y'Y^{-1} +Z'Z^{-1}),
        \\
        4\rho^{-2}_0(A\eta_\alpha(0)+B\eta_\beta(0)) &= X^{-1}\eta_1(0) + Y^{-1}\eta_2(0) + Z^{-1}\eta_3(0)
        \\ 
        &\hphantom{=}~+ 4\rho^{-4}_0B^2(X^2+Y^2+Z^2) + 4\rho^{-2}_0(A'^2+B'^2)
        \\
        &\hphantom{=}~+4\rho^{-2}_0A'(X+Y+Z) + 4A\rho^{-2}_0(X'+Y'+Z')
        \\
        &\hphantom{=}~-(X'^2X^{-2}+Y'^2Y^{-2}+Z'^2Z^{-2}).
    \end{align*}
\end{lemma}
\begin{proof}
From~\eqref{SigmaHodge - reducible}, we obtain 
\begin{align*}
    \|\omega_1\|^2 &= \rho_0^{-\frac{4}{3}}q^2(0)X^{-2}=X^{-2}, & \|\omega_2\|^2 &= \rho_0^{-\frac{4}{3}}q^2(0)Y^{-2}=Y^{-2}, & \|\omega_3\|^2 &= \rho_0^{-\frac{4}{3}}q^2(0)Z^{-2}=Z^{-2},\\
    \|\omega_{12}\|^2 &= (XY)^{-2}, & \|\omega_{23}\|^2 &= (YZ)^{-2}, &
    \|\omega_{13}\|^2 &= (XZ)^{-2},\\
    \|\alpha\|^2 &=4\rho_0^{-2}, & \|\beta\|^2 &=4\rho_0^{-2},
\end{align*}
where $\|\cdot\|=\langle\cdot,\cdot\rangle^{\frac12}$. The expressions in the first condition in~\eqref{suff_con} become
\begin{align*}
    \langle\psi^+,\psi_{\scparal}\rangle &= \langle A'\alpha+B'\beta, A\alpha+B\beta\rangle= 4\rho_0^{-2}(AA'+BB'),\\
    \langle\psi_-,\psi_{\scperp}\rangle &= \langle X'\omega_1+Y'\omega_2+Z'\omega_3, X\omega_1+Y\omega_2+Z\omega_3\rangle=X'X^{-1} + Y'Y^{-1} +Z'Z^{-1}.
\end{align*}
Integration over $\Sigma$ amounts to multiplication by the volume of~$\Sigma$. Therefore, the second formula in~\eqref{suff_con} holds point-wise.
On the left-hand side of this formula,
\begin{align*}
    \langle\eta,\psi\rangle - 2\langle\eta_{\scperp},\psi_{\scperp}\rangle &= \langle\eta_{\scparal},\psi_{\scparal}\rangle - \langle\eta_{\scperp},\psi_{\scperp}\rangle
    \\
    &= \langle \eta_\alpha(0)\alpha+\eta_\beta(0)\beta, A\alpha + B\beta\rangle
    \\
    &\hphantom{=}~- \langle\eta_1(0)\,\omega_1 + \eta_2(0)\,\omega_2 + \eta_3(0)\,\omega_3, X\omega_1 + Y\omega_2 + Z\omega_3\rangle
    \\
    &=4\rho_0^{-2}(A\eta_\alpha(0) +B\eta_\beta(0) ) - X^{-1}\eta_1(0) - Y^{-1}\eta_2 (0) - Z^{-1}\eta_3 (0).
\end{align*}
Analogously, using~\eqref{exterior derivative for reducible case} and~\eqref{SigmaHodge - reducible} yields
\begin{align*}
    \|d^\Sigma\psi_{\scparal}\|^2 + \|\delta^\Sigma\psi_{\scparal}\|^2 &= \|-2B(\omega_{12} + \omega_{23} +\omega_{13})\|^2+\|2A\rho_0^{-\frac{2}{3}}q^{-2}(0)(X^2\omega_1+Y^2\omega_2+Z^2\omega_3)\|^2
    \\
    &=4\rho_0^{-2}(X^2+Y^2+Z^2),
    \\
    \|d^\Sigma\psi_{\scperp}\|^2 + \|\delta^\Sigma\psi_{\scperp}\|^2 &= \frac14\|(X+Y+Z)\alpha\|^2 = \rho_0^{-2}(X+Y+Z)^2,\\
    \|\psi^++d^\Sigma\psi_{\scperp}\|^2 &= \Big\|\Big(A'+\frac{1}{2}(X+Y+Z)\Big)\alpha+B'\beta\Big\|^2
    \\
    &= 4\rho_0^{-2}(A'^2+B'^2) + \rho_0^{-2}(X+Y+Z)^2 +4\rho_0^{-2}A'(X+Y+Z),
    \\
    \|\psi^--\delta^\Sigma\psi_{\scparal}\|^2 &= \|(X'-2\rho_0^{-2}AX^2)\,\omega_1+(Y'-2\rho_0^{-2}AY^2)\,\omega_2+(Z'-2\rho_0^{-2}AZ^2)\,\omega_3\|^2 \\
    &=(X'^2X^{-2}+Y'^2Y^{-2}+Z'^2Z^{-2}) + 4\rho_0^{-4}A^2(X^2+Y^2+Z^2)
    \\
    &\hphantom{=}~-4\rho_0^{-2}A(X'+Y'+Z').
\end{align*}
Thus, up to the volume of $\Sigma$, the right-hand side of the second formula in~\eqref{suff_con} equals
\begin{align*}
    4\rho_0^{-4}B^2&(X^2+Y^2+Z^2) + 4\rho_0^{-2}(A'^2+B'^2) + 4\rho_0^{-2}A'(X+Y+Z)
    \\
    &- (X'^2X^{-2}+Y'^2Y^{-2}+Z'^2Z^{-2}) + 4\rho_0^{-2}A(X'+Y'+Z'). 
\end{align*}
\end{proof}

In light of Lemma~\ref{lem_calcs_red}, the first two formulas in~\eqref{initial_conditions} become
\begin{align}\label{ini_fh_red}
\begin{split}
f_\alpha(0)=A,\quad f_\beta(0)=B,\quad h_1(0)&=X,\quad h_2(0)=Y,\quad h_3(0)=Z, 
\\
f_\alpha'(0)&=\frac12(X+Y+Z+2A'),\quad 
f_\beta'(0)=B'.
\end{split}
\end{align}
The third is equivalent to
\begin{align*}
    -\Big(\frac{4}{3}\rho^{\frac{1}{3}}\rho'q^{-2}h_1+\rho^{\frac{4}{3}}q^{-3}q'h_1-\rho^{\frac{4}{3}}q^{-2}h_1'-2\rho^{-\frac{2}{3}}q^{-2}h_1^2f_\alpha\Big)(0) = X',\\
    -\Big(\frac{4}{3}\rho^{\frac{1}{3}}\rho'q^{-2}h_2+\rho^{\frac{4}{3}}q^{-3}q'h_2-\rho^{\frac{4}{3}}q^{-2}h_2'-2\rho^{-\frac{2}{3}}q^{-2}h_2^2f_\alpha\Big)(0) = Y',\\
    -\Big(\frac{4}{3}\rho^{\frac{1}{3}}\rho'q^{-2}h_3+\rho^{\frac{4}{3}}q^{-3}q'h_3-\rho^{\frac{4}{3}}q^{-2}h_3'-2\rho^{-\frac{2}{3}}q^{-2}h_3^2f_\alpha\Big)(0) = Z'.
\end{align*}
Using~\eqref{Assumption on the metric - reducible case} and~\eqref{ini_fh_red}, we can transform this into
\begin{align*}
    \Big(\frac{4}{3}\rho\rho'+\rho^{\frac{4}{3}}q'-h_1'YZ\Big)(0)-2XA &= -X'YZ, \\
    \Big(\frac{4}{3}\rho\rho'+\rho^{\frac{4}{3}}q'-h_2'XZ\Big)(0)-2YA &= -XY'Z,\\
    \Big(\frac{4}{3}\rho\rho'+\rho^{\frac{4}{3}}q'-h_3'XY\Big)(0)-2ZA &= -XYZ'.
\end{align*}
Next, we add the three formulas together and observe that
\begin{align*}
q'=\frac{1}{3}q^{-2}(h_1h_2h_3)', \qquad
\rho'(0)=\frac{1}{2}\rho_0^{-1}(A(X+Y+Z+2A')+2BB'),
\end{align*}
to obtain
\begin{align*}
    (4\rho\rho'+(h_1h_2h_3)'-h_1'YZ-Xh_2'Z-XYh_3')(0)-2(X+Y+Z)A &= -X'YZ - XY'Z - XYZ'
\end{align*}
or, equivalently,
\begin{equation*}
    -4(AA'+BB') = X'YZ + XY'Z + XYZ'.
\end{equation*}
By Lemma~\ref{suff_con - reducible case}, this identity follows from our hypotheses. Thus, the Cauchy conditions~\eqref{initial_conditions} amount to~\eqref{ini_fh_red} and
\begin{equation}\label{ini_fh_red_cont}
\begin{split}
    h_1'(0)&= \frac{4}{3}\rho_0^{-1}\rho'(0)X+\rho_0^{-\frac{2}{3}}q'(0)X-2\rho_0^{-2}X^2A + X',\\
    h_2'(0)&= \frac{4}{3}\rho_0^{-1}\rho'(0)Y+\rho_0^{-\frac{2}{3}}q'(0)Y-2\rho_0^{-2}Y^2A + Y'.
\end{split}
\end{equation}

To prove the existence portion of Theorem~B, it suffices to solve system~\eqref{Poisson - reducible case} subject to~\eqref{ini_fh_red} and~\eqref{ini_fh_red_cont}. Because the matrix $M(h_1,h_2,h_3)$ in~\eqref{sys_matr_red} is degenerate, standard ODE results do not apply directly. As in Section~\ref{subsubsec_exist_irred}, we overcome this difficulty by using a variant of DeTurck's trick. Let us construct a $G_2$-structure
\begin{align*}
\hat\phi=\hat h_1\,dr\wedge\omega_1 +\hat h_2\,dr\wedge\omega_2 +\hat h_3\,dr\wedge\omega_3 +\hat f_\alpha\,\alpha+\hat f_\beta\,\beta
\end{align*}
in a neighbourhood of~$\Sigma$ such that
\begin{align}\label{DeTurck_red}
\begin{split}(\hat h_1\hat h_2\hat h_3)'(0)&=0,\quad (\hat h_1\hat h_2\hat h_3)''(0)\ne0,\\ \Delta_{\hat\phi}\hat\phi&=\Phi^*\eta,\qquad 
\hat\phi|_\Sigma=(\Phi^*\psi)|_\Sigma,\qquad 
i^*(\iota_{\hat{\mathbf n}} d\hat\phi)=\psi^+,\qquad (\hat\delta\hat\phi)_{\scparal}=\psi^-.
\end{split}
\end{align}
Here, $\Phi$ is the diffeomorphism between two neighbourhoods of~$\Sigma$ induced by $r\mapsto(\hat h_1\hat h_2\hat h_3)'(r)$, \linebreak $\hat{\mathbf n}=d\Phi^{-1}(\mathbf n)$, and $\hat\delta$ is the codifferential given by~$\hat\phi$.
The existence portion of Theorem~B will follow by setting $\phi=(\Phi^{-1})^*\hat\phi$.

The pullback $\Phi^*\eta$ satisfies 
\begin{align*}
\Phi^*\eta=(\hat h_1\hat h_2\hat h_3)''\hat \eta_1\,dr\wedge\omega_1 +(\hat h_1\hat h_2\hat h_3)''\hat \eta_2\,dr\wedge\omega_2 +(\hat h_1\hat h_2\hat h_3)''\hat \eta_3\,dr\wedge\omega_3 +\hat\eta_\alpha\,\alpha+\hat\eta_\beta\,\beta
\end{align*}
with
\begin{align*}
\hat \eta_i=\eta_i\circ(\hat h_1\hat h_2\hat h_3)',\qquad 
\hat\eta_\alpha=\eta_\alpha\circ(\hat h_1\hat h_2\hat h_3)',\qquad
\hat\eta_\beta=\eta_\beta\circ(\hat h_1\hat h_2\hat h_3)',\qquad  i=1,2,3.
\end{align*}
Substituting $\hat{\phi}$ for $\phi$ and $\Phi^*\eta$ for $\eta$ in~\eqref{Poisson - reducible case}, we find
\begin{equation}\label{transformed Poisson}
\begin{split}
-\hat{q}''\hat{q}^2+ \hat{h}_1''\hat{h}_2\hat{h}_3 &= \hat{\rho}^{-\frac{4}{3}}\hat{q}^2(\hat q^{3})''\hat{h}_2\hat{h}_3\hat{\eta}_1 - \hat{h}_1\hat{G}_{(i)} - \hat{G}_{(ii)} - \hat{h}_1'\hat{h}_2\hat{h}_3\hat{G}_{(iii)} - \hat{h}_1'\hat{G}_{(iv)},
\\
-\hat{q}''\hat{q}^2+ \hat{h}_1\hat{h}_2''\hat{h}_3 &= \hat{\rho}^{-\frac{4}{3}}\hat{q}^2(\hat q^{3})''\hat{h}_2\hat{h}_3\hat{\eta}_2 - \hat{h}_2\hat{G}_{(i)} - \hat{G}_{(ii)} - \hat{h}_1\hat{h}_2'\hat{h}_3\hat{G}_{(iii)} - \hat{h}_2'\hat{G}_{(iv)},
\\
-2(\hat{h}_1'\hat{h}_2'\hat{h}_3+\hat{h}_1'\hat{h}_2\hat{h}_3'+\hat{h}_1\hat{h}_2'\hat{h}_3') &= \hat{\rho}^{-\frac{4}{3}}\hat{q}^2(\hat q^{3})''(\hat{h}_2\hat{h}_3\hat{\eta}_1 + \hat{h}_1\hat{h}_3\hat{\eta}_2 + \hat{h}_1\hat{h}_2\hat{\eta}_3) + \hat{G}_*,
\\
 \hat{f}_\alpha'' &= - \hat{\rho}^{-\frac{4}{3}}\hat{q}^2\hat{\eta}_\alpha + \hat{G}_\alpha,
\\
\hat{f}_\beta'' &= - \hat{\rho}^{-\frac{4}{3}}\hat{q}^2\hat{\eta}_\beta+\hat{G}_\beta.
\end{split}
\end{equation}
Here, 
$$\hat q=(\hat h_1\hat h_2\hat h_3)^\frac13,\qquad \hat\rho=\sqrt{\hat f_\alpha^2+\hat f_\beta^2},$$
and $\hat G_{(i)}$, $\hat G_{(ii)}$, $\hat G_{(iii)}$, $\hat G_{(iv)}$, $\hat G_\alpha$, $\hat G_\beta$ and $\hat G_*$ are algebraic functions of $\hat h_1$, $\hat h_2$, $\hat h_3$, $\hat h_1'$, $\hat h_2'$, $\hat h_3'$, $\hat f_\alpha$, $\hat f_\beta$, $\hat f_\alpha'$ and $\hat f_\beta'$. Condition~\eqref{nondeg_assumtion} enables us to transform~\eqref{transformed Poisson} into a system with the vector of the leading terms equal to
\begin{align}\label{sys_DeTurck_matr_red}
\frac13N(\hat h_1,\hat h_2,\hat h_3)\left(
\begin{matrix}
\hat h_1'' \\
\hat h_2'' \\
\hat h_3'' \\
\hat f_\alpha'' \\
\hat f_\beta''
\end{matrix}\right),\qquad
N(x,y,z)=\left(
\begin{matrix}
2yz & -xz & -xy & \hphantom{x}0\hphantom{y} & \hphantom{x}0\hphantom{y} \\
-yz & 2xz & -xy & \hphantom{x}0\hphantom{y} & \hphantom{x}0\hphantom{y} \\
yz & xz & xy & \hphantom{x}0\hphantom{y} & \hphantom{x}0\hphantom{y} \\
0 & 0 & 0 & \hphantom{x}1\hphantom{y} & \hphantom{x}0\hphantom{y} \\
0 & 0 & 0 & \hphantom{x}0\hphantom{y} & \hphantom{x}1\hphantom{y} \\
\end{matrix}\right).
\end{align}
Since the matrix $N(\hat h_1,\hat h_2,\hat h_3)$ is non-degenerate, we are able to apply Picard--Lindel\"of theorem to this system. Before doing so, however, we need to supplement it with initial conditions.

Fix $\kappa>0$ and assume that
$$(\hat{h}_1\hat{h}_2\hat{h}_3)''(0)=\kappa>0.$$
The equations $\hat\phi|_\Sigma=(\Phi^*\psi)|_\Sigma$ and $i^*(\iota_{\hat{\mathbf n}} d\hat\phi)=\psi^+$ in~\eqref{DeTurck_red} hold if and only if
\begin{align}\label{ini_DeT_fh_red}
\begin{split}
\hat f_\alpha(0)&=A,\quad \hat f_\beta(0)=B, \quad \hat h_1(0)=\kappa\,X,\quad \hat h_2(0)=\kappa\,Y,\quad \hat h_3(0)=\kappa\,Z,
\\
\hat f_\alpha'(0)&=\frac12\kappa\,(X+Y+Z+2A'),\quad
\hat f_\beta'(0)=\kappa\,B'.
\end{split}
\end{align}
Observing that, by~\eqref{Assumption on the metric - reducible case},
\begin{equation*}
\hat{\rho}_0^2=(A^2+B^2) = XYZ = \kappa^3\hat{q}^3(0)
\end{equation*}
and computing as above, we re-state the condition $(\hat\delta\hat\phi)_{\scparal}=\psi^-$ as
\begin{equation}\label{eq_del_hat_cond}
\begin{split}
    \kappa^{-2}\hat{h}_1'(0)&= \frac{4}{3}\kappa^{-1}\hat{\rho}^{-1}_0\hat{\rho}'(0)X+\kappa^{-2}\hat{\rho}^{-\frac{2}{3}}_0\hat{q}'(0)X-2\hat{\rho}^{-2}_0X^2A + X',\\
    \kappa^{-2}\hat{h}_2'(0)&= \frac{4}{3}\kappa^{-1}\hat{\rho}^{-1}_0\hat{\rho}'(0)Y+\kappa^{-2}\hat{\rho}^{-\frac{2}{3}}_0\hat{q}'(0)Y-2\hat{\rho}^{-2}_0Y^2A + Y',\\
    \kappa^{-4}(\hat{h}_1\hat{h}_2\hat{h}_3)'(0) &= \kappa^{-2}(YZ\hat{h}_1' + XZ\hat{h}_2' + XY\hat{h}_3') (0) -X'YZ - XY'Z - XYZ' - 4(AA'+BB').
\end{split}
\end{equation}
Assume that 
\begin{equation*}
(YZh_1' + XZh_2' + XYh_3')(0) = 0.
\end{equation*}
This formula implies $(\hat{h}_1\hat{h}_2\hat{h}_3)'(0)=0$ and $\hat{q}'(0)=0$. Furthermore, when it holds, the third line in~\eqref{eq_del_hat_cond} follows from our hypotheses by Lemma~\ref{suff_con - reducible case}. Therefore, we can re-state conditions~\eqref{eq_del_hat_cond} as
\begin{equation}\label{DeTurck_ini_cont}
    \begin{split}
    \hat{h}_1'(0)&= \frac{4}{3}\kappa\hat{\rho}^{-1}_0\hat{\rho}'(0)X-2\kappa^2\hat{\rho}^{-2}_0X^2A + \kappa^2X',\\
    \hat{h}_2'(0)&= \frac{4}{3}\kappa\hat{\rho}^{-1}_0\hat{\rho}'(0)Y-2\kappa^2\hat{\rho}^{-2}_0Y^2A + \kappa^2Y',\qquad
    (YZh_1' + XZh_2' + XYh_3')(0)=0.
    \end{split}
\end{equation}
The Picard--Lindel\"of theorem yields functions $\hat h_1$, $\hat h_2$, $\hat h_3$, $\hat f_\alpha$ and $\hat f_\beta$ defined near $r=0$ and satisfying~\eqref{transformed Poisson}, \eqref{ini_DeT_fh_red} and~\eqref{DeTurck_ini_cont}. Our construction of $\hat\phi$, and hence the proof of the existence portion of Theorem~B, will be complete as soon as we validate our earlier assumption that $(\hat h_1\hat h_2\hat h_3)''(0)$ equals~$\kappa$.

\begin{lemma}\label{lem_big_beautiful_lemma}
If the functions $\hat h_1$, $\hat h_2$, $\hat h_3$, $\hat f_\alpha$ and $\hat f_\beta$ solve system~(\ref{transformed Poisson}) with initial conditions~(\ref{ini_DeT_fh_red}) and~(\ref{DeTurck_ini_cont}), then formulas~(\ref{suff_con}) imply
    \begin{equation*}
        (\hat{h}_1\hat{h}_2\hat{h}_3)''(0)=\kappa.
    \end{equation*}
\end{lemma}
\begin{proof}
The result will follow from the evaluation of the third equation in~\eqref{transformed Poisson} at $r=0$. Throughout the proof of the lemma, we omit the argument~0, thus writing, for example, $\hat\rho$ and $\hat q$ instead of $\hat\rho(0)$ and $\hat q(0)$. With this convention,
\begin{align*}
\hat{q}^3=\kappa^3\hat{\rho}^2=\kappa^3XYZ, \qquad \hat{q}'=0.
\end{align*}

Let us begin by finding $\hat{G}_*$. We multiply the first line in~\eqref{sys_raw_Poisson_red} by $YZ$, open the brackets, and exploit~\eqref{ini_DeT_fh_red} to obtain
\begin{align*}
-2\kappa XA' - \frac{4}{9}\kappa^{-1}(\hat{\rho}')^2 &- \frac{4}{3}\kappa^{-1}\hat{\rho}\hat{\rho}'' - \kappa^{-2}\hat{\rho}^{\frac{4}{3}}\hat{q}''
\\ 
&+ \kappa^{-2}\hat{h}_1''YZ-\frac{4}{3}\hat{\rho}^{-1}\hat{\rho}'AX + 2\hat{f}_\alpha'X + 4\kappa^{-1}\hat{h}_1'A= (\hat{h}_1\hat{h}_2\hat{h}_3)''\hat{\eta}_1YZ.
\end{align*}
Manipulating the other two lines similarly and adding the outcomes yields
\begin{align*}
-2\kappa(X&+Y+Z)A' - \frac{4}{3}\kappa^{-1}(\hat{\rho}')^2 - 4\kappa^{-1}\hat{\rho}\hat{\rho}'' - 3\kappa^{-2}\hat{\rho}^{\frac{4}{3}}\hat{q}'' \\
&+ \kappa^{-2}(\hat{h}_1''YZ+\hat{h}_2''XZ+\hat{h}_3''XY) - \frac{4}{3}\hat{\rho}^{-1}\hat{\rho}'A(X+Y+Z)
\\
&+ 2\hat{f}_\alpha'(X+Y+Z) + 4\kappa^{-1}A(\hat{h}_1'+\hat{h}_2'+\hat{h}_3') = (\hat{h}_1\hat{h}_2\hat{h}_3)''(\hat{\eta}_1YZ+\hat{\eta}_2XZ+\hat{\eta}_3XY).
\end{align*}
Observe that
\begin{equation*}
    \hat{q}''= \frac{1}{3}\hat{q}^{-2}(\hat{h}_1\hat{h}_2\hat{h}_3)''=\frac{1}{3}\hat{\rho}^{-\frac{4}{3}}(\hat{h}_1''YZ+\hat{h}_2''XZ+\hat{h}_3''XY + 2\kappa^{-2}(\hat{h}_1'\hat{h}_2'\hat{h}_3+\hat{h}_1'\hat{h}_2\hat{h}_3'+\hat{h}_1\hat{h}_2'\hat{h}_3')).
\end{equation*}
Consequently, the evaluation of the third equation in~\eqref{transformed Poisson} at $r=0$ reduces to the evaluation of
\begin{equation}\label{main}
\begin{split}
        -2\kappa A'(X&+Y+Z)  - \frac{4}{3}\kappa^{-1}(\hat{\rho}')^2 - 4\kappa^{-1}\hat{\rho}\hat{\rho}''
        \\
        &- 2\kappa^{-4}(\hat{h}_1'\hat{h}_2'\hat{h}_3+\hat{h}_1'\hat{h}_2\hat{h}_3'+\hat{h}_1\hat{h}_2'\hat{h}_3') -\frac{4}{3}\hat{\rho}^{-1}\hat{\rho}'A(X+Y+Z)
        \\
        &+ 2\hat{f}_\alpha'(X+Y+Z) + 4\kappa^{-1}A(\hat{h}_1'+\hat{h}_2'+\hat{h}_3') = (\hat{h}_1\hat{h}_2\hat{h}_3)''(\hat{\eta}_1YZ+\hat{\eta}_2XZ+\hat{\eta}_3XY).
\end{split}
\end{equation}
Also,
\begin{align*}
    \hat{G}_*(0)&=2\kappa^5(X+Y+Z)A' + \frac{4}{3}\kappa^3(\hat{\rho}')^2 + 4\kappa^3\hat{\rho}\hat{\rho}''
    \\
    &\hphantom{=}+\frac{4}{3}\kappa^4\hat{\rho}^{-1}\hat{\rho}'A(X+Y+Z) - 2\kappa^4\hat{f}_\alpha'(X+Y+Z) - 4\kappa^3A(\hat{h}_1'+\hat{h}_2'+\hat{h}_3').
\end{align*}

Our next step is to find $\hat G_\alpha$ and $\hat G_\beta$. From the fourth equation in~\eqref{sys_raw_Poisson_red},
\begin{align*}
-\frac{4}{3}\kappa\hat{\rho}^{-1}\hat{\rho}'(X+Y+Z) + (\hat{h}_1'+\hat{h}_2'+\hat{h}_3') &+ 2\kappa^2\hat{\rho}^{-2}A(X^2 + Y^2 + Z^2) 
\\
&- \frac{4}{3}\kappa\hat{\rho}^{-1}\hat{\rho}'A' + (\hat{h}_1'+\hat{h}_2'+\hat{h}_3') - 2\hat{f}_\alpha'' = 2\kappa^2\hat{\eta}_\alpha, \\
\end{align*}
Therefore,
\begin{align*}
\hat{f}_\alpha'' &= -\frac{2}{3}\kappa\hat{\rho}^{-1}\hat{\rho}'(X+Y+Z)
+ \kappa^2\hat{\rho}^{-2}A(X^2 + Y^2 + Z^2) - \frac{2}{3}\kappa\hat{\rho}^{-1}\hat{\rho}'A' + (\hat{h}_1'+\hat{h}_2'+\hat{h}_3') -\kappa^2\hat{\eta}_\alpha,
\\
\hat{G}_\alpha(0)&=-\frac{2}{3}\kappa\hat{\rho}^{-1}\hat{\rho}'(X+Y+Z)
+ \kappa^2\hat{\rho}^{-2}A(X^2 + Y^2 + Z^2) - \frac{2}{3}\kappa\hat{\rho}^{-1}\hat{\rho}'A' + (\hat{h}_1'+\hat{h}_2'+\hat{h}_3').
\end{align*}
Similarly,
\begin{align*}
\hat{f}_\beta'' &= -\frac{2}{3}\kappa\hat{\rho}^{-1}\hat{\rho}'B' + \kappa^2\hat{\rho}^{-2}B(X^2+Y^2+Z^2) -\kappa^2\hat{\eta}_\beta,
\\
\hat{G}_\beta(0) &= -\frac{2}{3}\kappa\hat{\rho}^{-1}\hat{\rho}'B' + \kappa^2\hat{\rho}^{-2}B(X^2+Y^2+Z^2).
\end{align*}

We now evaluate $(\hat{h}_1'\hat{h}_2'\hat{h}_3+\hat{h}_1'\hat{h}_2\hat{h}_3'+\hat{h}_1\hat{h}_2'\hat{h}_3')$. Observe that
\begin{align*}
    0&=\hat{q}^{-3}((\hat{h}_1\hat{h}_2\hat{h}_3)')^2 
    \\
    &=\hat{q}^{3}((\hat{h}_1^{-1}\hat{h}_1')^2+(\hat{h}_2^{-1}\hat{h}_2')^2+(\hat{h}_3^{-1}\hat{h}_3')^2) +2(\hat{h}_1'\hat{h}_2'\hat{h}_3+\hat{h}_1'\hat{h}_2\hat{h}_3'+\hat{h}_1\hat{h}_2'\hat{h}_3').
\end{align*}
From~\eqref{DeTurck_ini_cont},
\begin{align*}
    (\hat{h}_1^{-1}\hat{h}_1')^2 &= \Big(\frac{4}{3}\hat{\rho}^{-1}\hat{\rho}'-2\kappa\hat{\rho}^{-2}AX + \kappa X^{-1}X'\Big)^2\\
    &=\frac{16}{9}\hat{\rho}^{-2}(\hat{\rho}')^2+4\kappa^2\hat{\rho}^{-4}A^2X^2+\kappa^2X^{-2}X'^2
    \\
    &\hphantom{=}~-\frac{16}{3}\kappa\hat{\rho}^{-3}\hat{\rho}'AX+\frac{8}{3}\kappa\hat{\rho}^{-1}\hat{\rho}'X^{-1}X' - 4\kappa^2\hat{\rho}^{-2}AX'.
\end{align*}
One can obtain similar expressions for $(\hat{h}_2^{-1}\hat{h}_2')^2$ and $(\hat{h}_3^{-1}\hat{h}_3')^2$. Thus,
\begin{align*}
    -2\kappa^{-4}(\hat{h}_1'\hat{h}_2'\hat{h}_3&+\hat{h}_1'\hat{h}_2\hat{h}_3'+\hat{h}_1\hat{h}_2'\hat{h}_3')
    \\&= \frac{16}{3}\kappa^{-1}(\hat{\rho}')^2 + 4\kappa\hat{\rho}^{-2}A^2(X^2+Y^2+Z^2)\\
    &\hphantom{=}+ \kappa\hat{\rho}^2(X^{-2}X'^2+Y^{-2}Y'^2+Z^{-2}Z'^2)
    -\frac{16}{3}\hat{\rho}^{-1}\hat{\rho}'A(X+Y+Z)\\
    &\hphantom{=}+ \frac{8}{3}\hat{\rho}\hat{\rho}'(X^{-1}X'+Y^{-1}Y'+Z^{-1}Z')
    -4\kappa A(X'+Y'+Z').
\end{align*}

We are nearly ready to evaluate the third equation in~\eqref{transformed Poisson}. Clearly,
$$
\hat{\rho}''=-\hat{\rho}^{-1}(\hat{\rho}')^2+\hat{\rho}^{-1}((\hat{f}_\alpha')^2+(\hat{f}_\beta')^2+\hat{f}_\alpha \hat{f}_\alpha''+\hat{f}_\beta \hat{f}_\beta'').
$$
Combining this with~\eqref{ini_DeT_fh_red} and the formulas for $\hat f_\alpha''$ and $\hat f_\beta''$ produced earlier, we transform~\eqref{main} into
\begin{align*}
0 &= -4\kappa A'(X+Y+Z) + 8\kappa^{-1}(\hat{\rho}')^2 - 4\kappa( A'^2 + B'^2) - 4\hat{\rho}^{-1}\hat{\rho}'A(X+Y+Z) 
\\ 
&\hphantom{=}~+ \frac{8}{3}\hat{\rho}^{-1}\hat{\rho}'(AA'+BB')
+ 4\kappa A\hat{\eta}_\alpha + 4\kappa B\hat{\eta}_\beta - 4\kappa\hat{\rho}^{-2}B^2(X^2+Y^2+Z^2) 
\\
&\hphantom{=}~+ \kappa\hat{\rho}^2(X^{-2}X'^2+Y^{-2}Y'^2+Z^{-2}Z'^2)
+ \frac{8}{3}\hat{\rho}\hat{\rho}'(X^{-1}X'+Y^{-1}Y'+Z^{-1}Z') 
\\
&\hphantom{=}~- 4\kappa A(X'+Y'+Z') - (\hat{h}_1\hat{h}_2\hat{h}_3)''(\hat{\eta}_1YZ+\hat{\eta}_2XZ+\hat{\eta}_3XY).
\end{align*}
The first equation in Lemma~\ref{suff_con - reducible case} implies
\begin{equation*}
\frac{8}{3}\hat{\rho}\hat{\rho}'(X^{-1}X'+Y^{-1}Y'+Z^{-1}Z')= -\frac{32}{3}\hat{\rho}^{-1}\hat{\rho}'(AA'+BB').
\end{equation*}
With this,~\eqref{main} takes the form
\begin{align*}
0 &= -4\kappa A'(X+Y+Z) + 8\kappa^{-1}(\hat{\rho}')^2 - 4\kappa (A'^2 + B'^2) -4A\hat{\rho}^{-1}\hat{\rho}'(X+Y+Z) - 8\hat{\rho}^{-1}\hat{\rho}'(AA'+BB')
\\
&\hphantom{=}~+ 4\kappa A\hat{\eta}_\alpha + 4\kappa B\hat{\eta}_\beta - 4\kappa\hat{\rho}^{-2}B^2(X^2+Y^2+Z^2) + \kappa\hat{\rho}^2(X^{-2}X'^2+Y^{-2}Y'^2+Z^{-2}Z'^2)
\\
&\hphantom{=}~- 4\kappa A(X'+Y'+Z') - (\hat{h}_1\hat{h}_2\hat{h}_3)''(\hat{\eta}_1YZ+\hat{\eta}_2XZ+\hat{\eta}_3XY).
\end{align*}
Finally, 
$$
\hat{\rho}'=\kappa\hat{\rho}^{-1}\Big(\frac{1}{2}(X+Y+Z)A+AA'+BB'\Big).
$$
As a consequence,
\begin{align*}
8\kappa^{-1}(\hat{\rho}')^2 -4A\hat{\rho}^{-1}\hat{\rho}'(X+Y+Z)
-8\hat{\rho}^{-1}\hat{\rho}'(AA'+BB')=0,
\end{align*}
and~\eqref{main} becomes
\begin{align*}
    0 &= -4\kappa A'(X+Y+Z) - 4\kappa(A'^2 + B'^2) + 4\kappa(\hat{\eta}_\alpha A + \hat{\eta}_\beta B) -4\kappa\hat{\rho}^{-2}B^2(X^2+Y^2+Z^2)\\
    &\hphantom{=}~+ \kappa\hat{\rho}^2(X^{-2}X'^2+Y^{-2}Y'^2+Z^{-2}Z'^2) - 4\kappa A(X'+Y'+Z') 
    \\
    &\hphantom{=}~-(\hat{h}_1\hat{h}_2\hat{h}_3)''(\hat{\eta}_1YZ+\hat{\eta}_2XZ+\hat{\eta}_3XY).
\end{align*}
Together with Lemma~\ref{suff_con - reducible case} and the formula $\hat{\rho}^2=\rho^2=XYZ$, this identity implies that $(\hat{h}_1\hat{h}_2\hat{h}_3)''$ equals~$\kappa$.
\end{proof}

\subsubsection{Uniqueness}\label{subsubsec_uniq_red}

Let $\phi_1$ and $\phi_2$ be solutions to~\eqref{Poisson_equation}--\eqref{initial_conditions} in a neighbourhood of~$\Sigma$ with
\begin{equation*}
\phi_i=h_{1i}\,dr\wedge\omega_1+h_{2i}\,dr\wedge\omega_2+h_{3i}\,dr\wedge\omega_3+ f_{\alpha i}\,\alpha + f_{\beta i}\,\beta,\qquad i=1,2.
\end{equation*}
Choose $\kappa>0$ and consider the initial-value problems
\begin{align*}
h_{*i}''=\sqrt[3]{\frac{h_{*i}}{(h_{1i}h_{2i}h_{3i})\circ h_{*i}'}},\qquad h_{*i}(0)=\kappa^3 XYZ,\qquad h_{*i}'(0)=0,\qquad i=1,2.
\end{align*}
By the Picard--Lindel\"of theorem, there exists an interval $(-\epsilon,\epsilon)$ on which each of these problems has a unique solution. We choose $\epsilon$ small enough to ensure that $h_{*i}$ and $h_{*i}''$ are nonzero on~$(-\epsilon,\epsilon)$.
Denote by $\Phi_i$ the diffeomorphisms between neighbourhoods of $\Sigma$ induced by the maps $r\mapsto h_{*i}'(r)$. If
$$\hat{\phi}_i=\hat h_{1i}\,dr\wedge\omega_1+\hat h_{2i}\,dr\wedge\omega_2+\hat h_{3i}\,dr\wedge\omega_3+ \hat f_{\alpha i}\,\alpha + \hat f_{\beta i}\,\beta=\Phi_i^*\phi_i,\qquad i=1,2,$$ 
then
\begin{equation*}
\hat{h}_{1i}=(h_{1i}\circ h_{*i})\,h_{*i}'',\qquad \hat{h}_{2i}=(h_{2i}\circ h_{*i})\,h_{*i}'',\qquad \hat{h}_{3i}=(h_{3i}\circ h_{*i})\,h_{*i}''.
\end{equation*}
Consequently,
\begin{equation*}
\hat{h}_{1i}\hat{h}_{2i}\hat{h}_{3i}=((h_{1i}h_{2i}h_{3i})\circ h_{*i}')(h_{*i}'')^3 =((h_{1i}h_{2i}h_{3i})\circ h_{*i}') \frac{h_{*i}}{(h_{1i}h_{2i}h_{3i})\circ h_{*i}'}=h_{*i}.
\end{equation*}
Calculating as in the previous section, we conclude that the Cauchy problems
\begin{align*}
\Delta_{\hat\phi_i}\hat\phi_i&=\Psi_i^*\eta,\qquad 
\hat\phi_i|_\Sigma=(\Psi_i^*\psi)|_\Sigma,\qquad 
(d\hat\phi_i)_{\scperp}=\psi^+,\qquad (\hat\delta\hat\phi_i)_{\scparal}=\psi^-,\qquad i=1,2,
\end{align*}
are both equivalent to system~\eqref{transformed Poisson} with initial conditions~\eqref{ini_DeT_fh_red} and~\eqref{DeTurck_ini_cont}. The Picard--Lindel\"of theorem implies that $\hat{\phi}_1=\hat{\phi}_2$, which means $\Psi_1=\Psi_2$ and
\begin{equation*}
\phi_1=(\Psi_1^{-1})^*\hat\phi_1=(\Psi_2^{-1})^*\hat\phi_2=\phi_2,
\end{equation*}
in a neighbourhood of $\Sigma$.

\subsection{Projective orbits}\label{subsec_proj_orb}

Assume that $G=Sp(2)$ and $K=SU(2)U(1)$. The discussion on pages~207--208 of~\cite{CS02} demonstrates that, in this case, the Cauchy problem~\eqref{Poisson_equation}--\eqref{initial_conditions} reduces to~\eqref{Poisson - reducible case},~\eqref{ini_fh_red} and~\eqref{ini_fh_red_cont} with 
$\eta_2=\eta_3$, $Y=Z$ and $Y'=Z'$. The results of Sections~\ref{subsubsec_exist_red} and~\ref{subsubsec_uniq_red} yield the existence and uniqueness of $h_1$, $h_2$, $h_3$, $f_\alpha$ and $f_\beta$ satisfying ~\eqref{Poisson - reducible case},~\eqref{ini_fh_red} and~\eqref{ini_fh_red_cont} near $r=0$. An elementary standard argument shows that $h_2=h_3$.
This proves Theorem~B.

\section{Closed $G_2$-structures}\label{sec_closed}

Restricting one's attention to closed $G_2$-structures often leads to significant benefits in $G_2$-geometry. For example, such a restriction may simplify the analysis of the $G_2$-Laplacian and associated equations. While not elliptic in the traditional sense, this operator enjoys ellipticity in the direction of closed $G_2$-structures up to diffeomorphism~\cite{BX11,L20}. The corollary following Theorem~A demonstrates how our integral Gauss formula transforms when~$d\phi=0$. In this section, we discuss the solvability of the Poisson equation in the space of closed $G_2$-structures. Remarkably, we need to add one more formula to~\eqref{suff_con} in order to obtain an analogue of Theorem~B. More precisely, we have the following result.


\begin{theorem}\label{thm_closed}
Consider a manifold $M$ with closed $G_2$-structure~$\psi$ and a compact orientable embedded hypersurface $\Sigma$ in~$M$. Assume that the background metric $g_0$ coincides with the metric induced by~$\psi$. Given a closed 3-form $\eta$ on~$M$ and a (not necessarily closed) 2-form $\psi^-$ on~$\Sigma$, if the problem
\begin{align}\label{Poisson_ini_closed}
\Delta_\phi\phi=\eta,\qquad d\phi=0,\qquad \phi|_{\Sigma}=\psi|_{\Sigma},\qquad (\delta\phi)_{\scparal}=\psi^-,
\end{align}
has a solution in a neighbourhood of~$\Sigma$, then
\begin{equation}\label{sufcon_closed}
\begin{split}
\langle\psi^-,\psi_{\scperp}\rangle&=0,
\qquad \eta_{{\scparal}}=d^\Sigma\psi^-,
\\
\int_\Sigma \langle\eta,\nu\wedge\psi_{\scperp}+\psi_{\scparal}\rangle -2\int_\Sigma\langle \eta_{\scperp},\psi_{\scperp}\rangle
&=\big\|\delta^\Sigma\psi_{\scparal}\big\|_{L^2}^2-\big\|\delta^\Sigma\psi_{\scperp}\big\|_{L^2}^2
- 
\big\|\psi^-- \delta^\Sigma\psi_{\scparal}\big\|^2_{L^2}.
\end{split}
\end{equation}
If $M$ is cohomogeneity one with the symmetry group $G$ compact, connected and simple, $\Sigma$ is a principal orbit, the data $\eta$, $\psi$ and $\psi^-$ are $G$-invariant, and inequality~(\ref{nondeg_assumtion}) and formulas~(\ref{sufcon_closed}) hold, then $\Sigma$ has a neighbourhood in which problem~(\ref{Poisson_ini_closed}) possesses a unique \mbox{$G$-invariant} solution.
\end{theorem}

\begin{proof}
The first assertion of Theorem~\ref{thm_closed} follows from Corollary~\eqref{cor_ThmA} and from Remark~\ref{rem_2nd_order_closed}. To prove the second assertion, we consider the three possible orbit types as in Section~\ref{sec_Poisson}. If $G=G_2$ and $K=SU(3)$, then every closed form on $M$ must be coclosed; see Remark~\ref{rem_closed_irred}. Using the notation of Section~\ref{subsec_spherical} and calculating as in Section~\ref{subsubsec_exist_irred}, we find that
the equalities $d\eta=d\psi=0$ and~\eqref{sufcon_closed} imply
\begin{align*}
B=0,\qquad \eta_\beta(0)&=0,\qquad X'=0,\qquad \eta_\alpha(0)=0, \\
\eta_r(0)&=\frac43\rho_0^{-\frac23}(\rho_0^{-\frac43}X(A\eta_\alpha(0)+B\eta_\beta(0))-12\rho_0^{-2}XB^2)=0.
\end{align*}
However, this contradicts the assumption that $\langle \eta_{\scperp},\psi_{\scperp}\rangle$ does not vanish. Thus, the second assertion of Theorem~\ref{thm_closed} is vacuous.

Suppose that $G=SU(3)$ and $K=\mathbb T^2$. Since $\eta$ and $\phi$ are closed, the functions $\eta_{\beta}$ and $f_{\beta}$ are identically zero. With the notation of Section~\ref{subsec_flag_mf}, Lemma~\ref{lem_calcs_red} and Proposition~\ref{G2 Laplacian - reducible case} reduce problem~\eqref{Poisson_ini_closed} to the system
\begin{align}\label{sys_raw_Poisson_closed}
\begin{split}
h_1+h_2+h_3-2f_\alpha'&=0, \\
-(q^{-3}h_1^2((\rho^\frac{4}{3}qh_1^{-1})'-2\rho^{-\frac{2}{3}}q f_\alpha))'&=\eta_1,
\\
-(q^{-3}h_2^2((\rho^\frac{4}{3}qh_2^{-1})'-2\rho^{-\frac{2}{3}}q f_\alpha))'&=\eta_2,
\\
-(q^{-3}h_3^2((\rho^\frac{4}{3}qh_3^{-1})'-2\rho^{-\frac{2}{3}}q f_\alpha))'&=\eta_3,
\\
-\frac{1}{2}q^{-3}(h_1^2(\rho^\frac{4}{3}qh_1^{-1})' + h_2^2(\rho^\frac{4}{3}qh_2^{-1})' + h_3^2(\rho^\frac{4}{3}qh_3^{-1})')+\rho^{-\frac{2}{3}}q^{-2}f_\alpha(h_1^2 + h_2^2 + h_3^2)&=\eta_\alpha,
\end{split}
\end{align}
with Cauchy conditions
\begin{align}\label{inicon_ODE_closed}
\begin{split}
f_\alpha(0)=A,\quad h_1(0)=X,\quad h_2(0)=Y,\quad h_3(0)&=Z, 
\\
-\frac23A^{-1}(X+Y+Z)X+2A^{-1}X^2+\frac23h_1'(0)
-\frac13XY^{-1}h_2'(0)-\frac13XZ^{-1}h_3'(0)
&=X',
\\
-\frac23A^{-1}(X+Y+Z)Y+2A^{-1}Y^2-\frac13X^{-1}Yh_1'(0)
+\frac23h_2'(0)-\frac13YZ^{-1}h_3'(0)
&=Y',
\\
-\frac23A^{-1}(X+Y+Z)Z+2A^{-1}Z^2-\frac13X^{-1}Zh_1'(0)
-\frac13Y^{-1}Zh_2'(0)+\frac23h_3'(0)
&=Z'.
\end{split}
\end{align}
The vector of the leading terms of the first four equations in~\eqref{sys_raw_Poisson_closed} appears as
\begin{align*}
\tilde M(h_1,h_2,h_3)\left(
\begin{matrix}
f_\alpha' \\
h_1'' \\
h_2'' \\
h_3''
\end{matrix}\right),
\qquad
\tilde M(x,y,z)=
\begin{pmatrix}
 \hphantom{x}1\hphantom{y} & 0 & 0 & 0 \\
\hphantom{x}0\hphantom{y} & 2yz & -xz & -xy \\
\hphantom{x}0\hphantom{y} & -yz & 2xz & -xy \\
\hphantom{x}0\hphantom{y} & -yz & -xz & 2xy
\end{pmatrix}.
\end{align*}
Conditions~\eqref{inicon_ODE_closed} 
take the form
\begin{align*}
\left(
\begin{matrix}
f_\alpha(0) \\
h_1(0) \\
h_2(0) \\
h_3(0)
\end{matrix}\right)
=\left(
\begin{matrix}
A \\
X \\
Y \\
Z
\end{matrix}\right),\qquad \tilde M(X,Y,Z)\left(
\begin{matrix}
0 \\
h_1'(0) \\
h_2'(0) \\
h_3'(0)
\end{matrix}\right)
=
\left(
\begin{matrix}
0
\\
F_1 \\
F_2\\
F_3
\end{matrix}\right),
\end{align*}
with $F_1$, $F_2$ and $F_3$ dependent on $A$, $X$, $Y$, $Z$, $X'$, $Y'$ and~$Z'$. Reasoning as in Section~\ref{subsubsec_exist_red}, we can prove the existence of $f_\alpha$, $h_1$, $h_2$ and $h_3$ that solve the first four equations in~\eqref{sys_raw_Poisson_closed}, subject to~\eqref{inicon_ODE_closed}, near $r=0$. Let us show that these functions satisfy the fifth equation. This will immediately imply the solvability of~\eqref{Poisson_ini_closed}.

Adding together the second, third and fourth lines of~\eqref{sys_raw_Poisson_closed}, and recalling that $\eta$ is closed, yields
\begin{align*}
\begin{split}
-(q^{-3}(h_1^2(\rho^\frac{4}{3}qh_1^{-1})' + h_2^2(\rho^\frac{4}{3}qh_2^{-1})' + h_3^2(\rho^\frac{4}{3}qh_3^{-1})'))'+2(\rho^{-\frac{2}{3}}q^{-2}f_\alpha(h_1^2 + h_2^2 + h_3^2))'&=2\eta_\alpha'.
\end{split}
\end{align*}
Ergo, it suffices to show that the fifth equation in~\eqref{sys_raw_Poisson_closed} holds at $r=0$, i.e.,
\begin{align}\label{extracon_closed_raw}
\begin{split}
-\frac{1}{2}(XYZ)^{-1}(X^2(\rho^\frac{4}{3}qh_1^{-1})'(0) + Y^2(\rho^\frac{4}{3}qh_2^{-1})'(0) &+ Z^2(\rho^\frac{4}{3}qh_3^{-1})'(0))
\\ &+\rho_0^{-\frac{2}{3}}(XYZ)^{-\frac23}A(X^2 + Y^2 + Z^2)=\eta_\alpha(0).
\end{split}
\end{align}
We simplify and note that $X+Y+Z-2f_\alpha'(0)$ vanishes to conclude that~\eqref{extracon_closed_raw} is equivalent to
\begin{align*}
X'+Y'+Z'=2\eta_\alpha(0).
\end{align*}
However, this formula follows from the assumption $\eta_{{\scparal}}=d^\Sigma\psi^-$, i.e.,
\begin{equation*}
    \eta_\alpha(0)\alpha = \frac{1}{2}(X'+Y'+Z')\alpha.
\end{equation*}
This concludes the proof of the existence of the $G_2$-structure $\phi$ satisfying~\eqref{Poisson_ini_closed}. The uniqueness of $\phi$ follows from Theorem~B.

Finally, we can handle the last orbit type, i.e., the case where $G=Sp(2)$ and $K=SU(2)U(1)$, arguing as in Section~\ref{subsec_proj_orb}.
\end{proof}

\begin{cor}
Let $M$ be a cohomogeneity one manifold with compact, connected and simple symmetry group~$G$. Consider a closed $G$-invariant $G_2$-structure $\psi$ on~$M$, a compact orientable embedded hypersurface $\Sigma$ in~$M$, and a metric $g_0$ on $M$ as in Theorem~\ref{thm_closed}. Assume that $\Sigma$ is a principal orbit and $\phi$ is a $G$-invariant solution to the problem
\begin{align*}
\Delta_\phi\phi=\eta,\qquad \phi|_{\Sigma}=\psi|_{\Sigma},\qquad
(d\phi)_{\scperp}=0,\qquad (\delta\phi)_{\scparal}=\psi^-,
\end{align*}
in a neighbourhood of~$\Sigma$ with $d\eta=0$ and $\langle \eta_{\scperp},\psi_{\scperp}\rangle$ nonzero. The $G_2$-structure $\phi$ is closed if and only if $\eta_{{\scparal}}=d^\Sigma\psi^-$.
\end{cor}

\end{document}